\documentclass[11pt,a4paper]{article}
\usepackage{a4wide}
\usepackage{amsmath,amssymb,amsthm}
\usepackage[hidelinks]{hyperref}
\usepackage{enumerate}
\usepackage{soul}
\usepackage[normalem]{ulem}
\usepackage{xcolor}

\usepackage{mathdots}
\usepackage{url}
\usepackage{mathtools}
\usepackage{epsfig,amsfonts,graphicx,color}

\numberwithin{equation}{section}

\newtheorem{theorem}{\rm\bf Theorem}[section]

\newtheorem{lemma}[theorem]{\rm\bf Lemma}
\newtheorem{corollary}[theorem]{\rm\bf Corollary}
\newtheorem{remark}[theorem]{\rm\bf Remark}
\newtheorem{example}[theorem]{\rm\bf Example} 

\newtheorem{conjecture}[theorem]{\rm\bf Conjecture}

\newcommand{\rank}{\operatorname{rank}}
\newcommand{\R}{\mathbb{R}}

 \date{July 22, 2023}

{\small
\author{S. Bandyopadhyay,  B. Dacorogna,  V.S.  Matveev and M. Troyanov\medskip
\and S.B.:  Department of Mathematics \& Statistics, IISER Kolkata, Mohanpur 741246, India;\\ saugata.bandyopadhyay@iiserkol.ac.in\medskip 
\and V.M.: Institut f\"{u}r Mathematik, Friedrich-Schiller-Universit\"{a}t, 07737 Jena, Germany;
\and vladimir.matveev@uni-jena.de\medskip
\and  B.D. \&  M.T.: Institut  de Math\'{e}matiques EPFL, 1015 Lausanne, Switzerland;
\and   bernard.dacorogna@epfl.ch \ \&  \  marc.troyanov@epfl.ch 
}
}
\title{Bernhard Riemann 1861 revisited: existence of flat coordinates for an arbitrary bilinear form.}

\begin{document}

\maketitle

\bigskip

\begin{abstract} 
We generalize the celebrated results of  Bernhard Riemann and  Gaston  Darboux: we  give  necessary and sufficient conditions for a bilinear form  to be flat. 	More precisely, we give explicit necessary and sufficient conditions  for a  tensor field of type $(0,2)$ which is not necessary symmetric or skew-symmetric, and is possibly degenerate, to have constant entries in a local coordinate system.  \\

\end{abstract} 

\medskip

{\small \tableofcontents}

\section{Introduction.   }

In the  paper \cite{riemann2} of 1861 Bernhard Riemann   considered  what is now called a Riemannian metric, that is,  a symmetric positive definite 2-form $g=g_{ij}(x)$. He  asked and  answered  the question under what conditions there exists a  coordinate system such that $g$ is given by a constant matrix.  He proved that  such coordinates exist locally  if and only if   what is now called the Riemann curvature tensor is identically zero.  This result  was announced in Riemann's famous inaugural lecture  in 1854, see 
\cite[Abschnitt 4]{riemann1}. Both  the  inaugural lecture and the paper \cite{riemann2} are viewed  nowadays as the starting points of  Riemannian Geometry. 
Note that \cite{riemann2}  is written in Latin and its first part is not relevant to this question. An English translation of the relevant second part, with  a detailed discussion, can be found  in \cite[pp. 179--182]{Spivak}.  In particular it is explained there (and   was known before)  that the assumption of positive definiteness is not essential  for the proof of Riemann:  it is sufficient that the symmetric form is nondegenerate. See also \cite{riemann3}.

\medskip 

The case when the bilinear form is skew-symmetric was considered and solved by Gaston  Darboux \cite{Darboux}: he has shown that a nondegenerate differential 2-form $\omega= \omega_{ij}(x)$  is given by a constant matrix in a certain local coordinate system, if and only if it is closed. This result lays at the foundation of Symplectic Geometry. 

\medskip 

In the present paper we ask and give  a complete answer to the same question for an arbitrary  bilinear  form, that is a  tensor field of type $(0,2)$, which may have  nontrivial symmetric  and skew-symmetric parts  that can  be degenerate.  
Note that the   case where the symmetric part is nondegenerate  can  easily be   reduced to  the methods of Riemann (see e.g. \cite{BDMT} for a proof and a discussion of boundary, smoothness  and global issues). Indeed, the existence of   coordinates such that the components of the bilinear form  $g_{ij}  + \omega_{ij} $    are constant   implies the existence of a symmetric (torsion free) connection 
$\nabla = (\Gamma_{jk}^i)$ whose curvature is zero and such that the bilinear form is parallel. If the symmetric part   $g$ is nondegenerate,  the only candidate for the connection is the Levi-Civita connection; the necessary condition is then that  its curvature tensor vanishes.  The other necessary condition is that 
 the skew-symmeric part  $\omega$ is  parallel with respect to the Levi-Civita connection of $g$.
 These conditions are also  sufficient. Therefore, the   results  in the present paper are new only in  the case where  $g$ is degenerate and $\omega$ is arbitrary.  

\medskip 

Our results are formulated in a way that the hypothesis on $g$ and $\omega$ can effectively be checked using only differentiation and  algebraic manipulations, as was the case  in the results of  Riemann and Darboux (in particular, if the entries of the  bilinear forms  are explicitly given  by   elementary functions, or as solutions of explicit systems of algebraic equations with rational coefficients, 
then the necessary and sufficient conditions for the  the existence of flat coordinates can be checked using a computer algebra system).

 \medskip

Our paper is organized as  follows: in Section \ref{sec:2} we treat the case when $\omega=0$ and $g$ is (possibly) degenerate, see  Theorem \ref{thm:1} and Theorem \ref{thm:1b}.  In Section \ref{sec:3},  we consider in Theorems \ref{th.gplussymplectic} and  \ref{thm:2}   the case  where  the skew-symmetric part is nondegenerate; and the symmetric part may be degenerate. In Section  \ref{sec:4} we first treat the known  case 
when the symmetric part is zero (and the skew-symmetric part may be  degenerate), see Theorem \ref{thm:darboux},   and  then the  general case, when both  $g$ and $\omega$ are allowed to be degenerate, see Theorem \ref{thm:3}.  

Sections 2.1 and 3.2 are about regularity issues; the reader who is only interested in smooth tensors can ignore them without any loss.  Our proofs use a variety of ideas and methods coming from different areas of differential geometry and the final Section \ref{outlook} is an outlook of those methods.

 \medskip

Our investigation is  mostly local (with the exception of the global statements in Corollaries  \ref{cor:global} and  \ref{cor:global1}  and  the  related  global questions discussed  in the outlook   Section \ref{outlook}). Whenever possible, we give two proofs. The first proof   assumes that all objects  are sufficiently smooth, which allows for simpler and more geometric arguments and allows us to use  the simplest  possible mathematical language. Such proofs would be understood by  Bernhard  Riemann and mathematicians coming shortly after him, such as Sophus Lie, Gregorio Ricci-Curbastro, Gaston Darboux,  Tullio Levi-Civita   and Ferdinand Georg Frobenius.  
We recommend \cite[Chapters  4 and 5]{Spivak} or \cite[Chapters 3 and  4]{DFN} for some background on the  notations we use and relation to other notations commonly used in  differential geometry.  We also tried to give, whenever possible, a proof in a lower regularity.

 \section{The degenerate symmetric case. } \label{sec:2} 

We consider  a bilinear symmetric form $g=g_{ij}(x)$  and  call it a (possibly, degenerate) metric on a domain in 
$\mathbb{R}^n$ with coordinates $x^1,...,x^n$.  We view $g$ as a covariant tensor field, meaning that if $y^1,...,y^n$ are a different coordinate system, then 
in these coordinates $g$ has coefficients 
\begin{equation}\label{eq.chcorrdmetric}
  \tilde{g}_{ij}(y) = \sum_{r,s} g_{rs}(x) \frac{\partial x^r}{\partial y^i}  \frac{\partial x^s}{\partial y^j}.
\end{equation}
Here, and throughout the paper, unless otherwise specified, all indexes run from $1$ to $n$. A coordinate system is  called {\it flat}, if in this coordinate system $g$ is given by a constant matrix; our goal in this section  is to give necessary and sufficient conditions for  the existence of   local flat coordinate systems for a given degenerate metric $g$. Our first result will play a key role in building such coordinates.

\begin{theorem} \label{thm:1a}   
For every  $i,j,s$ consider 
\begin{equation}
\label{eq:gamma}  
\Gamma_{ij,s}: = \tfrac{1}{2}  \left( \tfrac{\partial  g_{js}  }{\partial x^i}  +   \tfrac{\partial  g_{is} }{\partial x^j } - \tfrac{\partial g_{ij}\ }{\partial x^s} \right)
\end{equation} 
(we call them  Christoffel symbols of the first kind). Then, at a point  $x$   there exist  numbers $\Gamma^i_{jk}$  with  $\Gamma^i_{jk}= \Gamma^i_{kj} $ (we call them  Christoffel symbols of the second  kind)  satisfying 
\begin{equation} \label{eq:2}
 \sum_s\left( \Gamma^s_{jk} g_{is} + \Gamma^s_{ik} g_{js}\right) = \frac{\partial  g_{ij}  }{\partial x^k} 
\end{equation}
  if and only if the following condition  holds:
\begin{equation} \label{eq:1} \sum_s \Gamma_{ij,s} v^s =0 \ \ \textrm{ for every  $v^s \in \mathcal{R}$,}\end{equation} 
where 
\begin{equation} \label{eq:Rg} 
  \mathcal{R}:= \mathcal{R}_g(x):= \mathrm{Kernel}(g) := \{ v \in T_xM \mid g(v , \ \cdot )=0\}.   
\end{equation} 
If such numbers $\Gamma_{jk}^i$ exist, the   ``freedom''  in choosing  them 
  is the addition   of possibly several terms  of  the form 
\begin{equation}\label{eq:freedom} 
v^iT_{jk} \textrm{    with $v \in \mathcal{R}$ and $T_{jk}=T_{kj}$.}
\end{equation}
   
Moreover, if the rank of $g$ is constant  and \eqref{eq:1}  holds for every point $x$, 
 then there exist smooth functions   $\Gamma^i_{jk}(x)$ with $\Gamma^i_{jk}= \Gamma^i_{kj} $  satisfying \eqref{eq:2}.  
\end{theorem} 

\medskip

\begin{proof}
We fix a point $x$  and 
 view  \eqref{eq:2} as a system of linear equations on   unknowns $\Gamma^i_{jk}$; the coefficients of this system come from $g$ and partial 
derivatives of $g$.  Remember now that a linear system of equation 
\begin{equation} \label{eq:2bis}
A y = b 
\end{equation}
 (where $A$ is a $N\times N$-matrix, $y= (y_1,.., y_{_N})$ is an  unknown vector and 
$b= (b_1,...,b_N) \in \mathbb{R}^N$ is a known vector) 
has a solution if and only if for every vector $a = (a_1,...,a_N) \in \mathbb{R}^N$ such that $a^t A=\vec 0$ we have $a^t b=0$. We observe that  the equation 
\eqref{eq:2} is of the form \eqref{eq:2bis} with 
$N= \frac{n^2(n+1)}{2}$. By standard algebraic manipulations (known at least to Levi-Civita) one reduces  \eqref{eq:2}   to the system of equations 
\begin{equation} \label{eq:4} 
\sum_s g_{sk}\Gamma^s_{ij} = \Gamma_{ij,k}. 
\end{equation}
Indeed,  replacing    $\sum_s g_{si}\Gamma^s_{jk}$ by $\Gamma_{jk,i}$ and $\sum_s g_{sj}\Gamma^s_{ik}$ by $\Gamma_{ik,j}$ in \eqref{eq:2} we see that any solution 
$\Gamma^i_{jk}$  of \eqref{eq:2} solves \eqref{eq:4} and vice versa, thus there are two equivalent linear systems. 
It remains to observe that the condition 
$a^t b=0$ applied to   \eqref{eq:4}   is just the condition \eqref{eq:1}, and then for a linear system of equations \eqref{eq:2bis} such that the coefficient matrix  $A$ 
and the free terms $b$  smoothly depend on $x$  one can find a smooth solution  provided a solution exists at every point and the rank  of $A$ is constant. 
\end{proof}

\smallskip 
 
 \textbf{Remark.}
The Christoffel  symbols $\Gamma^i_{jk}(x)$ from  the previous Theorem will always be considered to be the coefficients of an affine symmetric (torsion free)  connection.
This means that if $y^1, \dots, y^n$ is a different coordinate system, then the corresponding Christoffel symbols   $\tilde{\Gamma}^i_{jk}(y)$  should by definition be given by
\begin{equation*} 
\tilde{\Gamma}^{k}_{i j} (y) =  \sum_{a, b , c} \frac{\partial y^k}{\partial x^{c}}
\left({\Gamma}^{c}_{a b} (x)  \frac{\partial x^{a}}{\partial y^i} \frac{\partial x^{b}}{\partial y^j}
  +  \frac{\partial^2 x^{c}}{\partial y^i \partial y^j} \right).
\end{equation*}
This rule for the change of coordinate guarantees that the covariant derivative is a well defined operation on any tensor field, independently of the chosen coordinates,
that is if $P = P^{i_1...i_k}_{j_1...j_m}$ is a tensor field of type $(k,m)$, then
\begin{multline*}
\nabla_i P^{i_1...i_k}_{j_1...j_m}  =  \frac{\partial }{\partial x^i} P^{i_1...i_k}_{j_1...j_m}  \ + \   
\sum_s \left(P^{s i_2...i_k}_{j_1...j_m} \Gamma^{i_1}_{si} +   P^{i_1s i_3...i_k}_{j_1...j_m} \Gamma^{i_2}_{si} + \cdots  +P^{i_1i_3...i_{k-1}s}_{j_1...j_m} \Gamma^{i_k}_{si} \right)
 \\   - \sum_s\left(P^{i_1 ...i_k}_{sj_2...j_{m}} \Gamma_{ij_{1} }^s + P^{i_1 ...i_k}_{j_1sj_3...j_{m}} \Gamma_{ij_{2} }^s +\cdots  +P^{i_1 ...i_k}_{j_1...j_{m-1}s}\Gamma_{ij_{m} }^s\right)
\end{multline*}

is a well defined   tensor field of type $(k, m+1)$. This tensor field is called the \emph{covariant derivative} of $P$ and denoted by  $\nabla P$, and we say that $P$ is \emph{parallel}
if $\nabla P = 0$. For instance  \eqref{eq:2} just says that $g$ is  parallel with respect to $\nabla$.   The covariant derivative  depends on the freedom  \eqref{eq:freedom}, but by construction the condition $\nabla g = 0$ does not.

 \medskip

Our first main result is the following

\begin{theorem} \label{thm:1}
Suppose rank of $g$ is constant and assume  \eqref{eq:1} is fulfilled at any point. Then, for any 
smooth functions $\Gamma_{jk}^i$  with $\Gamma^i_{jk}= \Gamma^i_{kj} $  satisfying \eqref{eq:2} the functions 
\begin{equation} \label{eq:3bis}
 R_{ijk\ell}:= \sum_s g_{is} \left(\tfrac{\partial  }{\partial x^k} \Gamma^s_{j\ell} - \tfrac{\partial  }{\partial x^\ell} \Gamma^s_{jk}+  \sum_a \left(\Gamma^s_{ka} \Gamma^a_{\ell j} - \Gamma^s_{\ell a} \Gamma^a_{jk}\right)\right) 
\end{equation}  
do not depend on the freedom \eqref{eq:freedom}. Moreover, 
there exist flat coordinates  for $g$ if and only if   there exist smooth functions  $\Gamma^i_{jk}(x)$ with $\Gamma^i_{jk}= \Gamma^i_{kj} $  satisfying \eqref{eq:2}  such that\footnote{
We stress that, unless $g$ is non degenerate, Condition \eqref{eq:3} is  of course \emph{not} equivalent to the vanishing of
$
R_{jk\ell}^i = \tfrac{\partial  }{\partial x^k} \Gamma^i_{j\ell} - \tfrac{\partial  }{\partial x^\ell} \Gamma^i_{jk}+  \sum_a \left(\Gamma^i_{ka} \Gamma^a_{\ell j} -  \Gamma^i_{\ell a} \Gamma^a_{jk}\right) 
$. 
}
 \begin{equation} \label{eq:3}
 R_{ijk\ell}=0  \   \text{ for every $i,j,k,\ell$.}
\end{equation}
\end{theorem}

\begin{proof}  \label{sec:proof1}
 In order to show that    $R_{ijk\ell}$ does not depend on the freedom in choosing $\Gamma$, let us plug   $\tilde \Gamma^i_{jk}  = \Gamma^i_{jk} + v^iT_{jk}$    with $v \in \mathcal{R}$ instead of $\Gamma$
in  the formula   
\eqref{eq:3bis} for $R_{ijk\ell}$.  The terms of the form $v^s\tfrac{\partial} {\partial x^m} T_{jk}$,  
  $v^s \tfrac{\partial}{\partial x^k}  T_{jm},$ $v^sT_{ka}\tilde \Gamma_{\ell j}^a$,  $v^sT_{\ell a}\tilde \Gamma_{k j}^a$   vanish after contracting with $g_{is}$ so the result differs from the  initial formula for $R_{ijk\ell}$ by 
\begin{equation}  \label{eq:tmp} 
	\sum_s g_{is} T_{j\ell} \tfrac{\partial v^s }{\partial x^k}  +  \sum_{a} v^a \Gamma_{ka, i} T_{j\ell }  -  \sum_s g_{is} T_{k\ell} \tfrac{ \partial v^s}{\partial x^j}  - \sum_{a} v^a \Gamma_{\ell a, i} T_{jk}. 
\end{equation} 
Next, in view of condition $\sum_s g_{is}v^s=0$ we have that $\sum_s g_{is} T_{j\ell} \tfrac{ \partial v^s}{\partial x^k} = -\sum_s T_{j\ell}  v^s\tfrac{\partial g_{is}} {\partial x^k} $, which  together with \eqref{eq:gamma} imply that the sum of the  first two terms of  \eqref{eq:tmp} is equal to  $-\sum_s T_{j \ell} \Gamma_{ki,s} v^s\stackrel{\eqref{eq:1}}{=}0$. 
Similarly, the sum of the last  two   terms is zero. The argument proves that  the freedom in choosing  $\Gamma$ does not affect $R_{ijk\ell}$ and therefore the condition \eqref{eq:3}.

By the standard argument (due already to classics, see e.g.  \cite[Prop. 5 in Chapter  4]{Spivak}) we  know   that $R_{ijk\ell}$   is a tensor field.  Then, its vanishing in one coordinate system implies its vanishing in any other coordinate system. Then, the existence of  flat coordinates implies that $R_{ijk\ell}=0$, so  the conditions listed in Theorem \ref{thm:1} are  necessary. 

Let us prove that they are  sufficient. We first observe that for any smooth vector field $v \in \mathcal{R}$ the metric $g$ is preserved by its flow. 
Indeed, the Lie derivative of the metric is given by 
\begin{eqnarray*}
(\mathcal{L}_vg)_{ij}&=& \sum_s\left( v^s \tfrac{\partial g_{ij}}{\partial x^s}  +  g_{is} \tfrac{\partial v^s }{\partial x^j} + g_{js} \tfrac{\partial v^s }{\partial x^i}   
\right)  \\
 &=& \sum_s\left( v^s \tfrac{\partial g_{ij}}{\partial x^s}     - v^s\tfrac{\partial g_{is} }{\partial x^j} - v^s \tfrac{\partial g_{js} }{\partial x^i}\right) \\
  & = &   - 2  \sum_s v^s \Gamma_{ij,s} =0. 
\end{eqnarray*}

Next, let us show that the distribution $\mathcal{R}$ is integrable, that is, for any two vector fields $v,u$ 
 from this distribution its commutator $[u,v]$ lies in the distribution. We obtain it by direct calculations: 
\begin{eqnarray*}
\sum_i g_{ij} [u,v]^i & = &  \sum_{s,i} \left(g_{ij} u^s \tfrac{\partial v^i}{\partial x^s} - g_{ij} v^s \tfrac{\partial u^i}{\partial x^s} \right)= 
-\sum_{s,i} \left( v^iu^s  -u^i v^s \right) \tfrac{\partial g_{ij} }{\partial x^s}\\
 &=&\sum_{s,i} \left( v^iu^s  -u^i v^s \right) (\Gamma_{is,j}+\Gamma_{js, i}) = 0.
\end{eqnarray*} 
 Then, 
  there exist coordinates 
$(x^1,...,x^k, y^1,...,y^{n-k})$ such the distribution is spanned by $\tfrac{\partial  }{\partial y^1},..., \tfrac{\partial  }{\partial y^{n-k}}$. In these coordinates the metric has the form 
$$
g= \sum_{i,j}^{k} g_{ij} dx^i dx^j.   
$$
Since  the vector fields $\tfrac{\partial }{\partial y^i}\in \mathcal{R}$ and therefore their flows preserve $g$, 
 the components  $g_{ij}$ are independent of $y$-coordinates. We then may view $g$ as a metric on a $k$-dimensional manifold with local coordinate system $x^1,...,x^k$. 
Equation \eqref{eq:2} implies that $\left(\Gamma_{jm}^i\right)_{i,j,m=1,...,k}$ are coefficients of the Levi-Civita connection of this   metric (of dimension $k$).  
Without loss of generality, because of the freedom \eqref{eq:freedom}, we may assume that all $\Gamma^i_{jm}$ with $i> k$  are equal to zero. Then, the formula for 
the components $R_{ij\ell m}$ of the curvature tensor (with lower indexes) 
 of this $k$-dimensional metric coincides, for $i,j,\ell, m\le k$, 
  with \eqref{eq:3bis}.   Then, the problem is reduced to the case when $g$ is nondegenerate, which was already solved by Riemann (see e.g. \cite[\S 4.4.7]{riemann3}). \end{proof}

 As the following example shows, the condition  \eqref{eq:3}  almost everywhere does not imply that the rank of $g$ is constant. 
\begin{example}  We consider the function  
$$\phi: \mathbb{R}^2 \to \mathbb{R}, \  \ \phi(x,y)= x^2 + y^2$$ 
and as $g$ we take $d\phi^2$. Locally, in a neighbourhood of any point different from $(0,0)$ the degenerate metric $g$ has constant coefficients 
 in  any coordinate system such that $\phi$  is the first coordinate.    Its rank falls to zero at the point $(0,0)$ and is one otherwise. By direct calculation  one sees that any continuous  solution $\Gamma^i_{jk}(x)$ of \eqref{eq:2} (assuming  $\Gamma^i_{jk}(x)=\Gamma^i_{kj}(x) $)  is not bounded when approaches  $(0,0)$. 
 \end{example}  The example  can easily be generalised for any dimension and any rank. 
On the other hand, the existence of   continuous  functions $\Gamma_{jk}^i$ satisfying  \eqref{eq:2} implies that the rank of $g$ is constant.

\begin{remark}  \rm The book \cite{kupeli} of D. Kupeli   studies  degenerate metrics (Kupeli calls them ``singular metrics''),   the corresponding  affine connections  and  their  curvature tensors. The condition  \eqref{eq:1} is equivalent to the {\it stationarity condition} \cite[Def. 3.1.3]{kupeli}. This author did  not  study the existence of flat coordinates but the invariance of  $R_{ijk\ell}$ with respect to the freedom and the Condition  \eqref{eq:freedom}   are  implicitly contained  in  his book. 
\end{remark}

\medskip

\begin{corollary} \label{cor:0}  
Assume  $g$ admits flat coordinates.   Consider the following system of PDE:
\begin{equation} \label{eq:pde} 
0=  \nabla_j u_i:= \frac{\partial  u_i}{ \partial x^j} - \sum_s \Gamma^s_{ij}{u_s}  
\end{equation}
on the unknown functions   $u_1(x),...,u_n(x)$, 
where $\Gamma^s_{ij}$ is a (smooth)  solution of  \eqref{eq:2}.
Then, for every point $\hat x$ and for any initial data $(\hat u_1,...,\hat u_n)\in \mathbb{R}^n $   such that for every $v\in \mathcal{R}(\hat x)$ we have $\sum_s v^s \hat u_s=0$    there exists a unique solution $u_1,...,u_n$  of \eqref{eq:pde} with the initial conditions $ u_i(\hat x)= \hat u_i$.  
This solution has the property  $\sum_s v^s u_s=0$ at every $x$ and for every $v\in \mathcal{R}( x)$. 
Furthermore,  for any such a  solution $u_1,...,u_n$   the $1$-form $u_1dx^1+...+u_ndx^n$ is closed so there  exists locally a function $f$ such that $\tfrac{\partial f}{\partial x^i}=u_i$. 
Moreover, if  a solution  vanishes at one point, it vanishes at every point.
\end{corollary}

\begin{proof} 
The equation  \eqref{eq:pde}  means that the 1-form $u_1dx^1+...+u_ndx^n$ is parallel with respect to the connection $\nabla= (\Gamma^i_{jk})$. In particular, the equation  is invariant with respect to the coordinate changes. Because  $\mathcal{R}$ is invariant under parallel transport,   if 
$\mathcal{R}\subseteq \operatorname{Kernel}(u_1dx^1+...+u_ndx^n)$   at  the point
$\hat x$,  then $\mathcal{R}\subseteq \operatorname{Kernel}(u_1dx^1+...+u_ndx^n)$  at every point. 
 In the flat coordinates $x^1,...,x^n$ 
 such that $g= \sum_{s=1}^r \varepsilon_i (dx^i)^2$ (with $\varepsilon_i \in \{-1, \ 1\}$)   the equation   \eqref{eq:pde} reads $\tfrac{\partial u_i}{\partial x^j}=0$. 
  Then, if the initial data satisfy $\sum_s v^s \hat u_s=0$,  then for any solution  we have 
$u_{r+1}=...=u_n=0$ 
and first $r$ functions $u_1,...,u_r$ satisfy $\tfrac{\partial u_i}{\partial x^j}=0$ which implies that they are arbitrary constants. 
\end{proof} 

\medskip

\begin{remark}\label{Remarkrank1} \rm 
The case of rank one metric is special, the following statement is true: \textit{If $g$ has rank 1, then $g= \pm \theta \otimes \theta$ for a locally defined (non zero)  1-form $\theta$.  Furthermore \eqref{eq:1} is equivalent to  $d\theta = 0$, and this  holds if and only if $g$ admits flat coordinates}. 

\smallskip 

 To see this, recall that  $g_{ij}$ is symmetric of rank one if and only if there exists    $(a_1, \dots, a_n)$, non vanishing,   such that $g_{ij} = \pm a_ia_j$.  
Suppose  \eqref{eq:2} holds for  $g= \pm \theta \otimes \theta$ with  $\theta= a_1dx^1+...+a_ndx^n$. Clearly, in the  flat coordinate system for $g$ the components $a_i$ are constant and $\theta$ is closed. In the other  direction, if $\nabla g=0$ then $\nabla(\theta)=0$ implying $d\theta=0$.  
\end{remark}

\medskip

So far we have worked on (an open subset of) $\mathbb{R}^n$, but because the conditions \eqref{eq:2} and  \eqref{eq:3} are coordinate invariant, they
have a meaning globally on a smooth manifold $M$ and we can state the following 

\begin{corollary} \label{cor:global} 
If $M$ is smooth closed  manifold such that $H_{\textrm{\rm dR}}^1(M)=0$, then it does not admit a  degenerate metric $g_{ij}$ of constant  $\operatorname{Rank}(g)\ge 1$ such that $R_{ijk\ell}=0$.  
\end{corollary} 
\begin{proof}
If $H_{\textrm{dR}}^1(M)=0$, any closed 1-form is exact so the form $\sum_iu_idx^i$ given by  Corollary \ref{cor:0} is the  differential of a function. Then, it   vanishes  at the points where the function takes its extremal values which gives a contradiction. 
\end{proof} 

\begin{corollary} \label{cor:global1}
If the smooth closed manifold $M$ admits a degenerate metric $g$ of rank $1$ such that  \eqref{eq:1} holds, then $M$ or its double cover 
is a fiber bundle over a circle. 
\end{corollary}

\begin{proof}
By  Remark \ref{Remarkrank1}, we know that  locally  $g= \pm  \theta \otimes \theta$ for a nowhere vanishing closed $1$-form $\theta$. 
Then $\theta$ is either well defined globally on $M$, or it is well defined on  a double cover. The claim  follows then from  \cite[Theorem 1]{tischler}. 
\end{proof}

\subsection{Optimal $C^r$-regularity for Theorem  \ref{thm:1}. }

It is known that for a non degenerate metric,   the following optimal regularity holds: if $g$ is of class $C^r$ with $r\in \mathbb{N}$ and satisfies  \eqref{eq:1}  and\eqref{eq:3}, then there exist flat coordinate systems of class $C^{r+1}$
(if $r=1$, then the curvature has to be interpreted in the sense of distributions\footnote{Note that the a similar result cannot hold for $r=0$. In   \cite[\S 6]{CH}, E. Calabi and P.  Hartman have given an example of a continuous metric which is locally isometric to the Euclidean metric but admits no flat coordinates of class $C^1$.}). We refer to  \cite{Cristinel}  or \cite[Theorem 8 and Remark 9]{BDMT} for a proof of this  optimality result.  
In the degenerate case, our proof of Theorem \ref{thm:1}  loses one degree of regularity when we ``factor out'' the kernel of $g$. Thus our proof of  Theorem \ref{thm:1}  assumes $g$ to be of class  $C^r$ with $r\geq 2$ and produces a flat coordinate system  of class $C^{r}$.  Our next result  states 
the existence of flat coordinates in optimal regularity:

\begin{theorem}  \label{thm:1b}   Suppose  $g$ has constant rank  and assume  \eqref{eq:1} holds  at any point. If  $g\in C^{r}$  for some $r \in \mathbb{N}$, then one can find $\Gamma_{jk}^i$ of class  $C^{r-1}$ such that  $\Gamma^i_{jk}= \Gamma^i_{kj} $  and \eqref{eq:2} holds. 
Moreover,  there exist   flat coordinates of  class   $C^{r+1}$ if and only if   \eqref{eq:3} is fulfilled.  
\end{theorem} 

\begin{remark} \rm 
\label{Remarque: R(G)=0 sens distrib} In our convention  the set $\mathbb{N}$ starts with $1$.  When $r=1,$ the condition
(\ref{eq:3}) has to be understood in the weak sense, see \cite[\S VI.I.6]{Hartman1982}. In the present situation, this  conditions means that for any 
$k, \ell \in \{1,\dots, n\}$ and any smooth $1$-form $u= (u_i)= \sum_i u_idx^i$  with  compact support such that $\mathcal{R}_g\subseteq \operatorname{Kernel}(u)$, 
we have
$$  
 \int \sum_s\left( -\Gamma^s_{j\ell} \tfrac{\partial u_s }{\partial x^k}  + \Gamma^s_{jk} \tfrac{\partial  u_s}{\partial x^\ell} +
  u_s \sum_a \left(\Gamma^s_{ka} \Gamma^a_{\ell j} - 
  \Gamma^s_{\ell a} \Gamma^a_{jk}\right)\right)dx = 0.
$$ 
This condition  is independent of  the freedom \eqref{eq:freedom}. 
\end{remark}

\medskip

\begin{proof} The proof that \eqref{eq:3} holds if there exist flat  coordinates is similar to the proof of the analogous statement in Theorem \ref{thm:1}. Also the proof that $\Gamma_{jk}^i$  can be chosen of regularity  $C^{r-1}$ is the same as in  Theorem \ref{thm:1}. In order to prove the existence and  smoothness of 
flat coordinates  assuming \eqref{eq:3},  let us consider a $n\times (n-m)$-matrix-valued function  $B(x)$ such that its columns are basis vectors of $\mathcal{R}_g$. Since $\mathcal{R}_g$ is given by a system of linear equations of constant rank whose coefficients are of class  $C^r$,   we may assume that $B$ is of regularity  $C^r$. Next, without  loss of generality we may assume that the last $n-m$ rows of $B$ form a nondegenerate matrix (of dimension $(n-m)\times (n-m)$). Then, there exists a  unique  $m\times (n-m)$-matrix-valued   
function  $F$  such that for every $x$ the vector $(u_1,...,u_n) $ whose first components $u_1,...,u_m$  are arbitrary and the other components $u_{m+1},...,u_n$ are constructed by $u_1,...,u_m$ via matrix-multiplication 
\begin{equation} \label{eq:F} 
  (u_{m+1},...,u_n)= (u_1,...,u_m) F
\end{equation}
 the following condition\footnote{Geometrically,  $(u_1,...,u_n) $  should be viewed as a covector, i.e., as  the 1-form 
$u_1dx^1 +...+u_ndx^n$. The condition \eqref{eq:condB} is just the condition  $\operatorname{Kernel}(u_1dx^1 +...+u_ndx^n)\supseteq \mathcal{R}_g$. }   is fulfilled: 
\begin{equation} \label{eq:condB}
  (u_1,...,u_n) B= 0.
\end{equation}
The matrix $F$  can be explicitly constructed   as follows: if we denote by $B'$ the submatrix of $B$ containing the first $m$ rows  of $B$ and by $B''$ the submatrix of $B$  containing the  last  $n-m$ rows by $B''$, then $B''$ is an invertible square matrix by hypothesis and $F$ is explicitly  given by $F= -B'(B'')^{-1}$.  In particular $F$ is of class $C^r$.

\medskip 

In what follows we denote the $i^{\textrm{th}}$ component of  the left hand side of \eqref{eq:F} by $F(u)_{m+i}$.
and we consider the following system of  $m\times n$  PDEs on $m$ unknown functions $u_1,...,u_m$ of  the variables $(x^1, \dots, x^n)$:
\begin{equation}\label{eq:pfaff}
\frac{\partial u_i}{\partial x^j} = \sum_{s=1}^m \Gamma^s_{ij} u_s  + \sum_{s=m+1}^n \Gamma^s_{ij} F(u)_s.
\end{equation} 
where  $1 \leq i \leq m$ and $1 \leq j \leq n$.   It  follows from \eqref{eq:freedom} that the system   \eqref{eq:pfaff} is independent of the choice of connection $\Gamma_{ij}^{k}$ satisfying  \eqref{eq:2}.
We observe the following facts concerning the system  \eqref{eq:pfaff}:
\begin{enumerate}[(i)]
 \item The  system \eqref{eq:pfaff} is  of Pfaff-Frobenius-Cauchy type, in the sense that all derivatives of unknown functions are linear expressions of unknown functions whose coefficients 
are functions of the position.
 \item If $g$ is of class $C^r$, with $r \geq 1$, then the coefficient of \eqref{eq:pfaff} are of class  $C^{r-1}$. This is due to the fact that  \eqref{eq:2} is a linear system  of constant rank with coefficient of class $C^r$  (see the  proof of Theorem \ref{thm:1a} for an explanation). One can therefore find $\Gamma^i_{jk}$  of  class  $C^{r-1}$ satisfying \eqref{eq:2}.

 \item The compatibility conditions  for  (\ref{eq:pfaff}) are equivalent to  \eqref{eq:3} (see e.g., \cite[\S VI.I.6]{Hartman1982}).
 \item If the compatibility conditions are satisfied, there exists, for any  point $p$  and any initial condition  $u_1(p),...,u_m(p)$, a unique (local) solution of \eqref{eq:pfaff} with this initial condition. Furthermore, if  the coefficients of  \eqref{eq:pfaff} are of class $C^{r-1}$ for some  $r\in \mathbb{N}$, then  this solution is  of class  $C^{r}$ (if $r=1$ the compatibility condition has to be interpreted n the weak sense). This statement is proved in  \cite[Chap. VI, Corollary 6.1]{Hartman1982}.  
\end{enumerate}

\medskip 		
		
Let us now show that if 	$u_1(x),...,u_m(x)$ is a solution of  \eqref{eq:pfaff}, then  the differential form  whose first  components   are $u_1(x),...,u_m(x)$ and the remaining $(n-m)$ components are  given by  \eqref{eq:F} is parallel with respect to any symmetric connection  $\nabla=  (\Gamma^i_{jk})$  whose coefficients satisfy \eqref{eq:2}.  
Indeed. for   $i \in \{1,...,m\}$  the condition $\nabla_j u_i= 0$ is clearly equivalent to \eqref{eq:pfaff}.  To deal with the case  $i \in \{m+1,...,n\}$ we need the following additional statement: for any   vector field $v= v^i\in \mathcal{R}_g$ of class $C^r$  and any vector field $z^j$   we have
\begin{equation} \label{eq:extend}
\sum_j  \left(z^j \nabla_j v^i\right) \ \in \  \mathcal{R}_g.		
\end{equation}
Indeed, 
$$
  0 = \sum_j z^j \tfrac{\partial }{\partial x^j}\left( \sum_{s,r}  g_{sr}   v^s w^r\right)=  \sum_{s,r, j }  g_{sr} w^r z^j  \nabla_j v^s +   
 \underbrace{\sum_{s,r,j}   g_{sr} v^s    z^j\nabla_j w^r}_{=0 \ \textrm{for $v\in \mathcal{R}_g$}} 
$$
so   $ \sum_j z^j \nabla_j v^s$ is a linear combination of the vectors from $\mathcal{R}_g$.  

Using  \eqref{eq:extend}, we obtain that for any $v\in \mathcal{R}_g$  and any $z$ (both of class  $C^r$) we have 
$$\sum_{i,j} v^i z^j \nabla_j u_i  = \sum_{j}  z^j \nabla_j\left( \sum_i  u_i v^i\right) - \sum_{i,j} u_i  z^j \nabla_j v^i = 0-0=0.$$ 
Then,  the covector   whose components are given by 
$$
  (\sum_j z^j \nabla_j u_1,..., \sum_j z^j \nabla_j u_n)
$$
satisfies  \eqref{eq:condB}, so its last $n-m$ components are determined by its first $m$ components via   \eqref{eq:F}. Since the first $m$ components are zero, as we proved above, also the last   $n-m$ components are zero.

Thus, we have shown that for any point $p$ and for any  initial values $u_1(p),...,u_n(p)$ such that   $\operatorname{Kernel}(u_1dx^1 +...+u_ndx^n)\supseteq \mathcal{R}_g(p)$ there exists a unique  $1$-form 
$u_i(x)= u_1(x)dx^1 +...+ u_n(x)dx^n$ of class  $C^r$ such that it is $\nabla-$parallel, moreover,  this form has the condition  $\operatorname{Kernel}(u_1dx^1 +...+u_ndx^n)\supseteq \mathcal{R}_g(p)$ at every point. This form is automatically closed.  We take $m$ linearly independent  1-forms of such type and denote by $f^1,...,f^m$ their primitive functions. At  the point $p$, there exists a $m\times m$ symmetric nondegenerate 
matrix $c_{ij}$ such that at $p$ we have $g= \sum_{i,j=1}^m c_{ij} df^i df^j.$ Since    by construction $g$ and each of the forms $df^i$ are parallel, this condition holds at any point so every  coordinate system such that the first $m$ coordinates are the  functions $f^1,...,f^m$ is flat for this metric.   
\end{proof}

\begin{corollary}  \label{cor:1b} 
Suppose $g$ has constant rank and satisfies  \eqref{eq:1} everywhere. Suppose also  $g\in C^{r,\alpha}$ with $r\in \mathbb{N}$ and $0\le \alpha\le 1$, then  there exists   a    flat coordinate system of class $C^{r+1,\alpha}$ if and only if  \eqref{eq:3} holds.
\end{corollary} 

\begin{proof} Arguing as in proof of Theorem \ref{thm:1b}, we consider the system \eqref{eq:pfaff},  whose solutions correspond to the differentials of the first $m$ flat coordinates.  We know that the solutions are of class $C^{r}$. We also see that 
 the derivatives of the solutions are  linear expression in the solutions with   coefficients  at least of class  $C^{r-1,\alpha}$. 
Therefore, the derivatives of the solutions are of class $C^{r-1,\alpha}$ and the solutions of \eqref{eq:pfaff} are therefore of  class $C^{r,\alpha}$.  This implies that the flat coordinates are of class $C^{r+1,\alpha}$. 
\end{proof} 

\medskip 

\begin{remark}\label{remarkflatfjs} \rm 
The proof of  Theorem \ref{thm:1b} shows that the metric $g$ has a flat coordinate system if and only there exist functions 
$f^1,...,f^m$ (with $m= \operatorname{Rank(g)}$) such that $g= \sum_{ij=1}^m c_{ij} df^i df^j$  with constant $c_{ij}$,
furthermore the $1$-forms $df^i$ are parallel and the flat coordinate system $x^1, \dots, x^n$  can be chosen such that $x^i = f^i$
for $1 \leq i \leq m$. Furthermore, if $g$ of of class $C^{r,\alpha}$, then on can chose $f^i$ of class  $C^{r+1,\alpha}$
\end{remark}

\section{On  flat coordinates  for the pair (degenerate metric, symplectic structure).   } \label{sec:3}

\subsection{Existence of  flat coordinates.}  

In this section we   obtain necessary and sufficient conditions for the existence of flat coordinates for the bilinear form $g+ \omega$ with nondegenerate skew-symmetric part $\omega$.  Obvious necessary conditions are that $g$ has flat coordinates and $\omega$ is a closed form.  
We will prove the following result:
\begin{theorem} \label{th.gplussymplectic}
Let $g$ be a a symmetric (possibly degenerate)  bilinear  form  such that there exist flat coordinates for it and $\omega= \omega_{ij}$  be a symplectic form.
Then, there exists a  coordinate system such that the components of both $g$ and $\omega$ are constant if and only if the equation 
 \begin{equation} \label{eq:parallel1}
    \sum_{a,b, c,d} g_{ia} P^{ab} P^{cd } g_{dj} \nabla_k \omega_{bc}    = 0
\end{equation}
holds  for every $i,j,k$, where  $P^{ij}$  is the inverse matrix of $\omega_{ij}$
$$
  \sum_s P^{is}\omega_{sj} = \delta_j^i  = \begin{cases} 1 & \text{ if } i = j, \\ 0 & \text{ if } i \neq j,\end{cases}
$$
and $\nabla$ is any connection compatible with $g$, i.e. satisfying \eqref{eq:2}. Condition \eqref{eq:parallel1} does not depend on the chosen connection.
\end{theorem}
 
\medskip

\begin{remark}
\emph{(i)} The matrix $P =  P^{ij}$, inverse of $\omega_{ij}$ represents a   \textit{contravariant} tensor field. This means that 
under a change of coordinates, the transformation rule is given by the rule dual to \eqref{eq.chcorrdmetric}:
$$
  \tilde{P}^{ij}(y) = \sum_{r,s} P^{rs}(x) \frac{\partial y^i}{\partial x^r}  \frac{\partial y^j}{\partial x^s}.
$$
\emph{(ii)} Another possible formulation of Condition  \eqref{eq:parallel1} can be written using  the $(1,1)$ tensor $J$ such that 
 $g(X,Y) = \omega (JX,Y)$, that is  $J^j_i = -\sum_k g_{ik}P^{kj}$.  Using this tensor, we  define a differential (skew symmetric) 2-form $\alpha$ by  
 $
   \alpha(u,v) := g(u,Jv).
 $
Condition  \eqref{eq:parallel1}  is then equivalent to $\nabla \alpha = 0$.
\end{remark}

Before proving Theorem \ref{th.gplussymplectic}, we first give  necessary and sufficient conditions for the existence of  a  local coordinate system  in which  both $g$ and $P$ have constant components:

\begin{theorem} \label{thm:2}
Let $g=g_{ij}$ be a symmetric (possibly degenerate)  bilinear  form  such that there exist flat coordinates for it  near a point  $p\in \mathbb{R}^n$  and $P= P^{ij}$  be a skew-symmetric tensor field of  rank  $n$  at  $p$. 
Then, there exists a  local coordinate system  near $p$ such that both $g$ and $P$ have constant components if and only if the following  conditions hold: 
\begin{enumerate}[(1)]
\item   $P^{ij}$ generates a Poisson structure, that is 
\begin{equation}  \label{eq:poisson}
\sum_s P^{sk} \tfrac{\partial }{\partial x^s}  P^{ij}   +  P^{si} \tfrac{\partial }{\partial x^s}  P^{jk} + P^{sj} \tfrac{\partial }{\partial x^s}  P^{ki}=0.
\end{equation} 
\item The following    holds for every $i,j,k$:  
\begin{equation} \label{eq:parallelPg}
\sum_{a,b} g_{ai} g_{bj}  \frac{\partial P^{ab}}{\partial x^k}  + \sum_s \left( P_i^{\  s} \Gamma_{ks, j}  +      P^{s}_{\ j}  \Gamma_{ks, i} \right) =0,
\end{equation} 
where  $ P_i^{\  s} = \sum_c g_{ic} P^{cs}$ and  $  P^{s}_{\ j}   = \sum_c g_{jc} P^{sc},$
 and $\Gamma_{ij,s}$ are as  in  \eqref{eq:gamma}. 
\end{enumerate}
\end{theorem}   

\medskip 

\begin{proof} 
Let us first observe that  conditions (\ref{eq:poisson}) and  (\ref{eq:parallelPg}) are geometric. Indeed 
 (\ref{eq:poisson}) is just the condition that the bilinear operation $\{ \cdot, \cdot \}$ defined on functions by 
\begin{equation}\label{pb} 
 \{ f, h\} := \sum_{i,j} \frac{\partial f}{\partial x^i} \frac{\partial h}{\partial x^j}P^{ij},
\end{equation}
satisfies the Jacobi identity and is therefore  a Poisson bracket. 
The condition  (\ref{eq:parallelPg})  says that   the tensor obtained by lowering both upper indexes in 
\begin{equation}\label{eq.nablaP}
  \nabla_k P^{ij} = \frac{ \partial } {\partial x^k} P^{ij} + \sum_s  \left(\Gamma^i_{sk} P^{sj} + \Gamma^j_{sk} P^{is}\right),
 \end{equation}
by $g_{ij}$ vanishes. 
In  particular (\ref{eq:parallelPg})   does not depend on the choice 
of the connection $\Gamma_{jk}^i$ satisfying  \eqref{eq:2}. Furthermore, both  (\ref{eq:poisson}) and (\ref{eq:parallelPg}) are tensorial conditions,
that are obviously satisfied in a flat coordinate system. So if there exists  flat coordinates for both $g$ and $P$, then 
then (\ref{eq:parallelPg}) and (\ref{eq:poisson}) hold in any coordinate system.

\medskip 

In order to prove Theorem \ref{thm:2} in the other direction, let us consider   smooth 
functions $f^1,..., f^m$  such that $g= \sum_{i=1}^m  \varepsilon_{i} (df^i)^2$ with $\varepsilon_1,...,\varepsilon_m\in \{-1, \ 1\}$.
 We assume  that the differentials of these functions are linearly independent in every points which implies 
$m= \operatorname{rank}(g)$. Furthermore $\nabla (df^i)= \left( \nabla_k\tfrac{\partial f^i}{\partial x^j}\right)=0.$
The existence of such functions  follows from the existence of flat coordinates.

We claim that  \eqref{eq:parallelPg} is equivalent to the condition  that for any $i,j \in \{1, \dots, m\}$ 
the Poisson bracket $\{ f^i, f^j\}$ is a constant. \ Indeed, using \eqref{eq.nablaP} and 
$$
 \nabla_k \left( \frac{\partial f^i}{\partial x^a}\right) = \frac{\partial^2 f^i}{\partial x^k\partial x^a} - \Gamma^b_{ka}\frac{\partial f^j}{\partial x^b} = 0,
$$
we obtain
\begin{align*}
  \nabla_k \{f^i, f^j\}  &=  \sum_{a,b} \nabla_k \left(P^{ab} \frac{\partial f^i}{\partial x^a} \frac{\partial f^j}{\partial x^b} \right) 
\\ &= \sum_{a,b} \nabla_k (P^{ab}) \frac{\partial f^i}{\partial x^a} \frac{\partial f^j}{\partial x^b}
+ \sum_{a,b} P^{ab} \nabla_k \left(\frac{\partial f^i}{\partial x^a}\right) \frac{\partial f^j}{\partial x^b}
+ \sum_{a,b}  P^{ab} \frac{\partial f^i}{\partial x^a} \nabla_k \left(\frac{\partial f^j}{\partial x^b}\right)
\\& = 0.
\end{align*}

Next, consider the vector fields $X_{f^1},..., X_{f^m}$ whose components are given by:
$$
X_{f^j}^{ \ \ i}= \sum_s P^{si} \tfrac{\partial f^j}{\partial x^s},
$$ 
(they are called the \textit{Hamiltonian vector fields of $f^j$}).  The condition \eqref{eq:poisson} implies  that they commute. Indeed, the commutator of the vector fields $X_{f^{\mu}} $ and $X_{f^\nu}$ is given by 
\begin{eqnarray*}  
[ X_{f^{\mu}}, X_{f^{\nu}}]^i &=&  \sum_{a,b,s}\left( P^{as} \tfrac{\partial f^{\mu}}{\partial x^a}   \tfrac{\partial }{\partial x^s}\left( P^{bi} \tfrac{\partial f^{\nu}}{\partial x^b}\right) - P^{as} \tfrac{\partial f^{\nu}}{\partial x^a}   \tfrac{\partial }{\partial x^s}\left( P^{bi} \tfrac{\partial f^{\mu}}{\partial x^b}\right)\right) 
\\&=& 
\sum_{a,b,s}\left( P^{as} \tfrac{\partial f^{\mu}}{\partial x^a}     P^{bi} \tfrac{\partial^{\nu} f^{\nu}}{\partial x^b \partial x^s }     + P^{as} \tfrac{\partial f^{\mu}}{\partial x^a}   \tfrac{\partial  P^{bi}}{\partial x^s}  \tfrac{\partial f^{\nu}}{\partial x^b}   - P^{as} \tfrac{\partial f^{\nu}}{\partial x^a}     P^{bi} \tfrac{\partial^{\nu} f^{\mu}}{\partial x^b \partial x^s }     -  P^{as} \tfrac{\partial f^{\nu}}{\partial x^a}   \tfrac{\partial  P^{bi}}{\partial x^s}  \tfrac{\partial f^{\mu}}{\partial x^b} \right)  \\
  &\stackrel{\eqref{eq:poisson}}{=}& \sum_s  P^{si} \tfrac{\partial  }{\partial x^s} \{ f^{\mu}, f^{\nu}\} = 0. 
\end{eqnarray*} 

Let us show that  there exists a function $f^{m+1}$    such that the differential $df^{m+1}$ is linearly independent (at the  point in whose small neighbourhood we are  working in) from the differentials of the functions $df^{1},...,df^{m}$ and such that 
for every $i= 1,...,m$  the function 
$
df^{m+1}(X_{f^i})
$
is  a constant.  In order to do it, we consider the coordinates $(t^1,...,t^m,z^{m+1},...,z^{n})$ such that in these coordinates for every $i=1,...,m$ 
the vector field $X_{f^{i}}$ is equal to $\tfrac{\partial}{\partial t^i}$.  The coordinates exist by the (simultaneous)  Rectification Theorem.

Chose now an arbitrary  1-form $\theta$ with constant entries in this coordinate system which is linearly independent from $df^1,...,df^m$.
Clearly  $d\theta$ is closed and we can choose $f^{m+1}$ such that  $df^{m+1} = \theta$. It is clear from the construction that 
$ \{f^{m+1},f^i\}  = \theta (X_{f^i})$ is constant for all $j$. 

We consider then the symmetric   bilinear  form  
$$g_{ext} :=  g  +  (df^{m+1})^2.$$ 
It has constant rank equal to $m+1$ and its entries are constant  in the coordinate system $(x^1= f^1,..., x^{m+1}=f^{m+1},x^{m+2},..., x^n)$. Moreover, 
 the (natural analog of the) condition \eqref{eq:parallelPg} is satisfied for this metric.  Indeed, this condition is equivalent to the condition that  
$$ \{f^i, f^j \} = \sum_{a,b}  \frac{\partial f^i}{\partial x^a}  \frac{\partial f^j}{\partial x^b}P^{ab}  $$
is constant for every $i,j=1,...,m+1$,  which is clearly the case by the construction. 

Then, we can enlarge  the rank of $g$ further and in $n-m$ such steps  come to the coordinate system $f^1,...,f^n$ in which both  the metric  and the tensor $P$  have  constant components. 
\end{proof} 

\medskip 

We can now prove the main Theorem of the section.
   
\begin{proof}[Proof of Theorem \ref{th.gplussymplectic}]
It is well  known that the dual $P$  of a symplectic form $\omega$ is a Poisson structure, 
thus condition  (\ref{eq:poisson}) is satisfied. We claim that \eqref{eq:parallel1} and \eqref{eq:parallelPg}  are equivalent conditions.
To prove this claim,   recall that  $\delta_j^i = \sum_s P^{is}\omega_{sj}$, is a parallel tensor  for any connection, therefore
$$
 0 = \nabla_k  (\delta_t^a)  = \nabla_k   (\sum_s P^{as}\omega_{st})  =  
  \sum_s (\nabla_k  P^{as})\omega_{st}  +  \sum_s P^{as} \nabla_k   \omega_{st},
$$
we thus have 
$$
   \nabla_k   P^{ab} = \sum_{t,s} P^{bt} (\nabla_k  P^{as})\omega_{st} =
   - \sum_{t,s} P^{as}P^{bt} \nabla_k   \omega_{st}.
$$
Lowering  both upper indexes in this identity by $g$  gives the equivalence \eqref{eq:parallel1} $\Leftrightarrow$ \eqref{eq:parallelPg}.
Theorem \ref{thm:2} gives us now  the existence of coordinates in which both $g$ and $P$ have constant entries. Clearly $\omega$ is also
constant in these coordinates.
\end{proof} 

\medskip 

The following example provides a simple instance where Theorem   \ref{thm:2} implies the existence of flat coordinates for $g+\omega$. However, directly establishing the existence of such coordinates may not be straightforward.
\begin{example} \rm 
Let us consider the following tensors in $\R^4$ :
$$
  g = (dx^1)^2 + (dx^2)^2 \quad \text{and} \quad  \omega = dx^1\wedge dx^2 + a \cdot dx^2\wedge dx^3 +dx^3\wedge dx^4,  
$$
where  $a = a(x^2,x^3)$ is a smooth, non constant function of $x^2$ and $x^3$.  Since $g$ is constant we will choose $\nabla$ to be the standard connection on $\R^4$.  A tensor is then parallel for $\nabla$ if and only its entries are constant. 
In matrix notations, the tensors $g$, $\omega$ and $P$ are 
$$
 G =  \left( \begin {array}{ rrrr} 1&0&0&0\\  0&1&0&0 \\  0&0&0&0\\  0&0&0&0\end {array} \right), \quad 
  \Omega =  \left( \begin {array}{ rrrr} 0&-1&0&0\\  1&0&-a&0 \\  0&a&0&-1\\  0&0&1&0\end {array} \right), 
 \quad \text{and} \quad  
  P = \Omega^{-1} =  \left( \begin {array}{ rrrr} 0&1&0&a\\  -1&0&0&0\\  0&0&0&1\\  -a&0&-1&0\end {array} \right). 
$$
The tensor 
$$
 GPG =  \left( \begin {array}{ rrrr} 0&1&0&0\\  -1&0&0&0 \\  0&0&0&0\\  0&0&0&0\end {array} \right) 
$$
is constant. By Theorem \ref{thm:2}, we know that  there exists a  local coordinate system  in some neighborhood of any point of  $\R^4$  such that $g$,  $P$ and $\omega$ have constant components. 

One should note however that  
$$
   J =  - PG = (GP)^{\top} =  \left( \begin {array}{ rrrr} 0&-1&0&0\\  1&0&0&0\\  0&0&0&0\\  a&0&0&0\end {array} \right).
$$
is \emph{not}   constant, hence not parallel for the connection $\nabla$. 
\end{example}

\medskip
 
\begin{remark} \label{rem:regularity} \rm 
In the proof of Theorem \ref{thm:2}, we have used several times the (simultaneous) \emph{Rectification Theorem}, which  states that if $X_1, \dots, X_k$ are $k$ linearly independent vector fields in a domain of $\R^n$ such that $[X_i,X_j] = 0$,  then there exist local coordinates $x^1, \dots, x^n$ in a neighborhood of any points such that $X_i = \frac{\partial}{\partial x^i}$ for 
$i = 1, \dots, k$. Furthermore, if the fields are of class  $C^{r,\alpha}$ with $r\ge 1$ and $0 \le \alpha\le 1$ , then the coordinates are  also of class 
$C^{r, \alpha}$. Indeed, by  standard results from the theory of ordinary differential equations, we know that
a  vector field of class $C^{r,\alpha}$ generates a flow of class   $C^{r,\alpha}$ (see e.g. \cite[Theorem 12.2]{book}). \\
Therefore, the proof of Theorem \ref{th.gplussymplectic}, shows that if one assumes that   $g$ and $\omega$ are of class  $C^{r,\alpha}$ with $r \geq \max \{(n-m), 1\}$, then there exists a   coordinate system of class  $C^{r+1 + m-n,\alpha}$ that is flat for  both $g$ and $\omega$.
The reason is that  in the proof of Theorem \ref{thm:2}, we loose one  class of regularity   at each step
of the construction (note that by Remark \ref{remarkflatfjs}, the functions $f^1, \dots, f^m$ are  of class  $C^{r+1, \alpha}$ and the proof requires $n-m$ steps, 
so the resulting coordinates are indeed of class $C^{r+1 + m-n,\alpha}$).  A better regularity result will be given in next section.  
\end{remark}

Finally, note that the  arguments in our previous proof also show that the following statement is true:

\begin{theorem}\label{Th.omegaandf}
Let $\omega$ be a symplectic form of class defined on a domain $U \subset \R^n$. Suppose there exists $f^1, \dots, f^m \in C^r(U)$
such that  $df^1, \dots, df^m$ are everywhere  linearly independent  and the Poisson brackets $\{f^i, f^j\}$ are constant on $U$ for any $i, j \in \{1, \dots, m\}$. If $r \geq p =n-m$, then there exists a coordinate system $y^1, \dots, y^n$ of class  in some neighborhood of any point in $U$ such that $y^i = f^i$ for $i = 1, \dots, m$ and $\omega$ has constant coefficients $\omega_{ij}$ in these coordinates.
\end{theorem}

Note in particular that the case $m = 0$   gives an alternative proof of Darboux' Theorem. We are not aware of such a proof in the literature.

\subsection{On the regularity  of flat coordinates for the pair (degenerate metric, symplectic structure). }

By Theorem \ref{thm:1b} and Corollary \ref{cor:1b}, if the (degenerate) metric $g$ is of class  $C^{r,\alpha}$, then the flat coordinate system, if it  exists, is of class  $C^{r+1,\alpha}$. A similar phenomenon  holds in the purely skew-symmetric case, when $g=0$ and $\omega$ is nondegenerate. Indeed, it has been proved in \cite[Theorem 18]{BD}  that given a symplectic form $\omega$  of class $C^{r,\alpha}$ with $0<\alpha<1$ and  $r\in \mathbb{N}\cup \{0\}$, there exists  local  coordinate systems of class  $C^{r+1,\alpha}$  in which $\omega$ has constant entries.
In view of these results, one might hope  that if $g$ and $\omega$ are of class $C^{r,\alpha}$, then a flat coordinate system of class $C^{r+1,\alpha}$ should exists for $g+\omega$. The following example ruins such hope.

\begin{example}\label{ex:2} 
We consider $\mathbb{R}^2$ with the  coordinates $(x,y)$ and the bilinear form $g+ \omega$ with 
$g= dx^2$ and $\omega= h(x)dx\wedge{dy}$ with $h\ne 0$. Then, the condition \eqref{eq:parallel1}
holds, and up to a $C^{r,\alpha}$-coordinate change, the flat coordinates are given by $(x, u(x,y))$ with the function $u$ satisfying the equation 
$\tfrac{\partial u}{\partial y}= h(x).$ The general solution of this equation is  $u(x,y)= \hat u(x) + y h(x)$ with an arbitrary function $\hat u(x)$. 
If $h$ is not of class $C^{r,\alpha}$, then  $u(x,y)$ is also not of class $C^{r,\alpha}$, which implies that flat coordinates cannot be of class  $C^{r+1,\alpha}$. 
\end{example}

The next result improves  Theorem \ref{th.gplussymplectic}: if the bilinear form is of class $C^{r,\alpha}$  with $3\le r \in \mathbb{N}$ and $0<\alpha<1$, then one can find  flat coordinates of class $C^{r-2,\alpha}$. 

\begin{theorem}  \label{thm:2bis} 
Under the hypothesis of Theorem \ref{th.gplussymplectic}, if the condition  (\ref{eq:parallel1})  is  fulfilled and the bilinear form $g+\omega$ is of class  $C^{r,\alpha}$  with $3\le r \in \mathbb{N}$ and $0<\alpha<1$, then  there exists a   flat coordinate system   of class     $C^{r-2,\alpha}$. 
\end{theorem}

\medskip

The rest of the section is devoted to proving this Theorem; the proof is quite involved and can be omitted with no damage
for the understanding of  the rest of the article.
For the proof, we will need the following two statements, which are known in folklore,   but for which we did not find explicit references. We sketch the ideas leading to the proof.

\begin{lemma}[Poincar\'e Lemma with parameters]  \label{PoincareParameter}
Let $\omega_s$  be a  family of  closed $m$-forms on a ball $U^n$ with coordinates $x^1,...,x^n$, where  $s=(s^1,...,s^k)$ are some parameters. 
Assume that the dependence of  the components of $\omega_s$ on $x$ and on $s$ is of class  $C^{r,\alpha}$ with $r\in \mathbb{N} $ and $0 \leq \alpha\leq 1$, then, there exists a  family $\theta_s$ of  $(m-1)$-forms,  such that their dependence on $(x,s)$ is of class  $C^{r,\alpha}$ and   such that for every $s$ we have  $d\theta_s= \omega_s$. 
\end{lemma} 

Indeed, the standard proof of the Poincar\' e Lemma (such as written in \cite{abraham}) is based on a purely algebraic construction
followed by an integration along a selected coordinate. The first operation obviously does not affect the regularity of the form with respect  
to any set of parameters and the integration also preserves the $C^{r,\alpha}$ regularity, thanks to the Lebesgue dominated convergence Theorem.

\begin{lemma}[Darboux Theorem with parameters] \label{DarbouxParameter}
Let $\omega_s$  be a  family of  symplectic  $2$-forms on a ball $U^{2n}$ with coordinates $x^1,...,x^{2n}$,  where  $s=(s^1,...,s^k)$ are some parameters. 
 Assume that the dependence of 
the components of $\omega_s$ on $x$ and on $s$ is of class  $C^{r,\alpha}$ with $r\in \mathbb{N} $ and  $0 \leq \alpha\leq 1$.
Then, there exists a  family $\phi_s$ of local  diffeomorphisms  $\phi_s: U^{2n}\to \mathbb{R}^{2n}$
 such that their dependence on $(x,s)$ is of class  $C^{r,\alpha}$ and such that for every $s$ the form $\omega_s$ is the pullback   under $\phi_s$
of the standard symplectic form on $\mathbb{R}^{2n}$.  
\end{lemma} 

Idea  of the proof: The proof via the ``Moser trick'' requires the Poincar\'e Lemma and standard facts about 
the existence and  regularity of systems of ordinary differential equations. This  allows one to keep track of how 
the change of coordinate system depends on the parameter $s$. Indeed, the  Moser trick is based on a construction of a (time depending) vector field such that its flow at time $t=1$  gives us the required diffeomorphism.  The construction of the vector field uses the Poincar\'e Lemma,  and applying the previous Lemma one can check that 
the vector field and its flow are of class $C^{r,\alpha}$ with respect to both the space variables $x$ and the parameter $s$.  See \cite[\S 3.2]{McDuffSalamon} for more details on Moser's proof. 

\medskip

 \begin{proof}[Proof of Theorem \ref{thm:2bis}]
By Corollary \ref{cor:1b} there exist functions $f^1,...,f^m$ of class  $C^{r+1,\alpha}$ 
with $m= \operatorname{Rank}(g)$ such that $g= \sum_{i,j=1}^m c_{ij} df^i df^j$ for some constant nondegenerate  symmetric $m\times m$-matrix $(c_{ij})$. By  (\ref{eq:parallelPg}),   the Poisson bracket 
of any two these  functions is constant. We may assume  without loss of generality that there exist $k',k''$ with $2 k'+k''=m$ such that 
$$
-\{f^i, f^{i+k'}\}=  \{f^{i+k'}, f^i\}=1 \ \textrm{for $i\le k'$}  
$$
and such that  for any other pair of functions $f^i$  its Poisson bracket
is zero. Next, as in Section \ref{sec:3},  we consider the  commuting  vector  fields $X_{f^i}$; they are of class   $C^{r,\alpha}$. By the Rectification Theorem, there exists a coordinate system  $(x^1,...,x^n)$  of class  $C^{r,\alpha}$ such that the following holds:
\begin{itemize} \item[(A)] The first $2k'+k''$ coordinates are $x^1=f^1,...,x^{2k'+k''}= f^{2k'+k''}$.
\item[(B)] The first $k'$ vector fields $X_{f^i}$, $i=1,...,k'$,  are given by: $X_{f^i}= -\tfrac{\partial }{\partial x^{k'+i}}$. 
\item[(C)] The next $k'$ vector fields $X_{f^i}$, $i=k'+1,...,2k'$,  are given by: $X_{f^i}= \tfrac{\partial }{\partial x^{i- k'}}$.
\item[(D)] The next $k''$   vector fields  $X_{f^i}$, $i=2k'+1,...,2k'+ k''$,  are given by: $X_{f^i}= -\tfrac{\partial }{\partial x^{i+ k''}}$.
\end{itemize} 

Let us explain the existence of this coordinate system. Consider the local 
action of $\mathbb{R}^{2k'+ k''}$ generated by the flows of commutative linearly independent 
vector fields $X_{f^1},...,X_{f^{2k'+k''}}$.  Take a transversal  $n-2k'-k''$-dimensional submanifold 
  to the orbits of this action  
  such that on this transversal  the values of $f^1,\dots, f^{2k''}$  are equal to zero.   We may  do it without loss of generality since adding a constant to $f^i$ changes nothing. 
	
 The functions $f^{2k'+ 1},...,f^{2k'+ k''}$ restricted to any transversal have linearly independent differentials since they 
 are constant on the orbits of the action of  $\mathbb{R}^{2k'+ k''}$.  We take a coordinate system  on the transversal such that its first $k''$ coordinates are $f^{2k'+1},...,f^{2k'+k''}$. 

Next, consider the  
 coordinates  $(t^1,...,t^{2k'+k''},y^{2k'+k''+1},...,y^n)$ 
 coming from the Rectification Theorem, constructed by these vector fields, by this transversal,  and by  this choice of the coordinates on the transversal. 
 Recall that these coordinates  have the following properties: The  vector fields $X_{f^i}$  are  the vectors $\tfrac{\partial }{\partial t^i}$. 

The coordinates   $(t^1,...,t^{2k'+k''},y^{2k'+k''+1},...,y^n)$, after the following  reorganisation and   proper changing the signs   are   as we require in (A--D):
 we consider  the coordinates
$$ 
\begin{array}{l}(x^1= t^{k'+1},...,x^{k'}=t^{2k'}, x^{k'+1}= -t^{1},...,x^{2k'}=-t^{k'},x^{2k'+1}= y^{2k'+k''+1},...,
x^{2k'+k''}= y^{2k'+2k''},\\ x^{2k'+k''+1}=-t^{2k'+1},...,
x^{2k'+k''}=- t^{2k'+k''},x^{2k'+2k''+1}=y^{2k'+2k''+1},..., x^n= y^n).\end{array}$$ 
Let us explain that 
 by the construction of the coordinates the first $m=2k'+ k''$  coordinates are the functions $f^1,...,f^m$ as we require in (A). Indeed, at our transversal the values of the coordinates $x_{1},...,x_{2k'}$ are zero and therefore  coincide with that of  $f^1,...,f^{2k'}$. Next, by the assumptions  
$$
  X_{f^i}(f^j)= P^{ij} =\{f^i, f^j\} =  
  \left\{ \begin{array}{rl} 
   -1 & \textrm{\ if \ $1 \leq i \leq k'$  and  $j = k' + i$,}\\ 
    1 & \textrm{\ if \ $1 \leq j \leq k'$  and  $i = k' + j$,}\\ 
    0 &  \textrm{otherwise.}\\ 
\end{array}\right.
$$
implying (A).  Finally, observe that the $i^{\textrm{th}}$ column    of $P$ is  the vector $-X_{x^i}$,  which  gives us 
 $(B,C,D)$.

\medskip

In this coordinate system the matrix of the Poisson structure $P$ is given as follows (since $P$ is skew-symmetric it is sufficient to 
 describe  the entries $P^{ij}$ with  $i> j$ only):  
Its first $k'$ columns  are the vectors $\tfrac{\partial }{\partial x^{k'+1}},...,\tfrac{\partial }{\partial x^{2k'}}$.  The next $k'$ columns 
are the vectors $-\tfrac{\partial }{\partial x^{1}},...,-\tfrac{\partial }{\partial x^{k'}}$. The next $k''$ columns
are  $\tfrac{\partial  }{\partial x^{2k'+k''+1}},...,\tfrac{\partial  }{\partial x^{2k'+2k''}}.$ 
Moreover, all entries of the matrix $P^{ij}$ areof class  $C^{r-1, \alpha}$  and independent of the coordinates $x^{1},...,x^{2k'}$ and of the coordinates $x^{m+1},...,x^{m+k''}$. 
Indeed, it is known and follows from the Jacobi identity that 
any  Poisson structure is preserved by the flow of any   Hamiltonian vector field. Then,  our Poisson structure  $P$ 
is preserved by the flows of the vector fields 
$\tfrac{\partial }{\partial x^1},...,\tfrac{\partial }{\partial x^{2k'}}$ and $\tfrac{\partial  }{\partial x^{2k'+k''+1}},...,\tfrac{\partial  }{\partial x^{2k'+2k''}}$ implying that its  entries are independent of  $x^{1},...,x^{2k'}$ and  of $x^{2k'+k''+1},...,x^{2k'+2k''}$.

\smallskip 

	For example, the general form  of such a matrix $P^{ij}$ with $k'= 1$, $k''= 2$, $n=8$ is as follows:
\begin{equation}\label{p8} 
P^{ij} = 	\begin{pmatrix} 0 & -1 & 0 & 0 & 0 & 0 & 0& 0 \\  
	                 1& 0 & 0 & 0 & 0 & 0 & 0& 0 \\
									0 & 0 & 0 & 0 & -1 & 0 & 0 &0 \\
									0 & 0 & 0 & 0 & 0 & -1 & 0 &0 \\
									0 & 0 & 1& 0 &  0 &    P^{56} & P^{57} & P^{58} \\
									0 & 0 & 0 & 1 &  -P^{56} &  0 & P^{67} &P^{68} \\
                  0 & 0&  0 & 0 & - P^{57}  & -P^{67} &  0 & P^{78}  \\ 
									0 & 0&  0 & 0 & -P^{58}  & -P^{68} &  -P^{78}  & 0  \end{pmatrix} 
\end{equation}
where the components $P^{ij}$ with $4<i<j\le 8$  are functions of the variables $x^3,x^4,x^7, x^8$ only. 
	
\medskip

Calculating the inverse matrix $(\omega_{ij})= (P^{ij})^{-1}$ we see that in these coordinates it is given by 
\begin{equation} \label{omega}
\begin{array}{rcl}
	\omega&=&\sum_{i= 1}^{k'}dx^i\wedge dx^{k' +i} + \sum_{i= 1+2k'}^{2k'+k''}dx^{i}\wedge dx^{k''+i}  +  \sum_{i, j=1+2k'}^{2k'+k''}   u_{ij} dx^{i} \wedge dx^{j} \\ &+&
	\sum_{i=2k'+1}^{2k'+k''}\sum_{\mu=2k'+2k''+1}^{n} v_{i\mu} dx^{i} \wedge dx^{\mu} + \sum_{\mu,\nu=2k'+2k''+1}^{n}   w_{\mu\nu} dx^{\mu}   \wedge dx^{\nu}.		
\end{array}
\end{equation}
The functions $u_{ij}$,
$ v_{i\alpha}$ and $w_{\alpha\beta}$ are explicit algebraic expressions in the entries of $P^{ij}.$  Therefore,
they are  of class   $C^{r-1,\alpha}$  and are  independent of the coordinates $x^{1},...,x^{2k'}$ and of the coordinates $x^{2k'+k''+1},...,x^{2k'+2k''}$. Note also that the  $(n-2k'-2k'') \times  (n-2k'-2k'')$-matrix $w_{\alpha\beta}$ is skew-symmetric and  nondegenerate.

For example, 	 the inverse of the matrix \eqref{p8} is as follows: 
	\begin{equation} \label{omegasmall} 
\omega_{ij} = 	\begin{pmatrix} 0 & 1 & 0 & 0  & 0 & 0 & 0& 0 \\  
	                              -1& 0 & 0 & 0  & 0 & 0 & 0& 0 \\
									              0 & 0 & 0 & u  & 1 & 0 & v_{37}  &v_{38} \\
									              0 & 0 &-u & 0  & 0 & 1 & v_{47}  &v_{48}  \\
									              0 & 0 &-1 & 0  & 0 &    0 & 0 &0 \\
									              0 & 0 & 0 & -1 &  0 &  0 & 0 &0 \\
                               0 & 0&  -v_{37}  & -v_{47} &  0  & 0 &  0 & w  \\ 
									              0 & 0&  -v_{38}  & -v_{48} & 0  & 0 &  -w   & 0  \end{pmatrix}, 
\end{equation} 
where the functions $v_{ij}$, $u$, $w$  may 
depend on $x^3,x^4,x^7, x^8$ only and are of class  $C^{r-1,\alpha}$.

Let us view now the  last sum in \eqref{omega}, namely 
 $\tilde \omega= \sum_{\alpha,\beta=2k'+2k''+1}^{n}   w_{\mu\nu} dx^{\mu}   \wedge dx^{\nu}, $ 
as a 2-form on a ($n-2k'-2k''$)-dimensional neighborhood with 
local coordinates  $(x^{2k'+2k''+1},...,x^n)$. The coordinates $(x^1,...,x^{2k'+2k''})$ are now viewed as parameters and  we actually know that the form $\omega$
does not depend on the coordinates $x^1,...,x^{2k'}$ and on the coordinates  $x^{2k'+k''+1},...,x^{2k'+2k''}$,  so effectively the parameters are $x^{2k'+1},...,x^{2k'+k''}$.  
The form $\tilde \omega$  is closed and non-degenerate; i.e., it is a symplectic  form.

By Lemma \ref{DarbouxParameter},   there exists a coordinate change (depending on the parameters)
\begin{equation} \label{coordinatechange}
\begin{array}{ll}	
x_{old}^{2k'+2k''+1} &= \ x_{new}^{2k'+2k''+1}(x^{2k'+1},...,x^{2k'+k''},x_{old}^{2k'+2k''+1},\dots, 	x_{old}^n),  \ \cdots 
\\    \\ x_{old}^{n}& =  \  x_{new}^{n}(x^{2k'+1},...,x^{2k'+k''},x_{old}^{2k'+2k''+1},...,x_{old}^n)
\end{array}
\end{equation}
such that after this coordinate change  $\tilde \omega$ has constant entries. Then, after the coordinate change of class $C^{r-1,\alpha}$ 
which leaves  the first coordinates $(x^{1},...,x^{2k'+2k''})$ unchanged  and transforms the remaining coordinates   $(x^{2k'+2k''+ 1},..., x^{n})$ by the rule \eqref{coordinatechange},  we achieve that the components  $w_{\mu, \nu }$ in 
\eqref{omega} are now constants. Note that this coordinate  change does  otherwise not affect the structure of $\omega$ given by \eqref{omega}.

\medskip

Next, we use that the form \eqref{omega} is closed.  The coefficient  of  $dx^{i} \wedge dx^\mu \wedge dx^\nu$  in $d\omega$ 
is given by $\left( \tfrac{\partial v_{i\mu}}{\partial x^\nu}- \tfrac{\partial v_{i\nu}}{\partial x^\mu}\right) $
for any $i \in \{2k'+1,...,2k'+k''\}$ and $\mu,\nu \in  \{2k'+2k''+1,\dots,n\}$. We thus see that for every $i= 2k'+1,...,2k'+k''$ the 1-form $\theta_i:= \sum_{\mu= 2k'+2k''+1}^n v_{i\mu} dx^\mu$, viewed as a 1-form on   a neighborhood of dimension $n-2k'-2k''$  with coordinates
 $(x^{2k'+2k''+1},...,x^n)$,  with coefficients depending on parameters $ x^{2k'+1},...,x^{2k'+k''}$, 	is closed. Then, by Lemma \ref{PoincareParameter}, there exist functions 
 $V_i$  of  class $C^{r-2,\alpha}$   such that $\tfrac{\partial V_i}{\partial x^\mu}= v_{i\mu}$.

\medskip 
 
We  consider now the following coordinate change of class  $C^{r-2,\alpha}$.  The coordinates $x^1,...,x^{2k'+k''}$ and $x^{2k'+2k''+ 1},..., x^{n}$ remain unchanged 
and the  coordinates $x^{2k'+k''+1},...,x^{2k'+2k''}$ are changed by the rule:
\begin{equation}\label{change2}
x_{old}^{2k'+k''+i}= x_{new}^{2k'+k''+i}+ V_i.
\end{equation} 
This coordinate change does not affect the previous improvements; that is: The structure of $\omega$ is still given   by \eqref{omega} and the terms $w_{\mu\nu}$  are still constant. But  now the terms $v_{i\mu}$ are zero.

Finally, we consider the term  $\sum_{i,j=2k'+1}^{2k'+k''} u_{ij} dx^i \wedge dx^j$ of \eqref{omega}.  Calculating the differential of the 2-form $\omega$ we see that 
$u_{ij}$ may depend on the coordinates $x^{2k'+1},...,x^{2k'+k''}$ only, and that the 2-form $\sum_{i,j=2k'+1}^{2k'+k''} u_{ij} dx^i \wedge dx^j$ viewed as a 2-form on a $k''$-dimensional neighborhood with coordinates  $x^{2k'+1},...,x^{2k'+k''}$ is closed.  
 The components of this form are   of class  $C^{r-3,\alpha}$, and 
by  \cite[Theorem  8.3]{book},   there exist functions $U_{2k'+1},...,U_{2k'+k''}$,  of class  $C^{r-2,\alpha}$, depending on the coordinates  $x^{2k'+1},...,x^{2k'+k''}$  such that 
$$
   \sum_{i,j=2k'+1}^{2k'+k''} u_{ij} dx^i \wedge dx^j = \sum_{i =2k'+1}^{2k'+k''} d \left( U_i dx^i  \right).   
 $$
We use the functions $U_i$  in the  last coordinate change: the coordinates $x^1,...,x^{2k'+k''}$ and $x^{2k'+2k''+ 1},..., x^{n}$ remain unchanged 
and the  coordinates $x^{2k'+k''+1},...,x^{2k'+2k''}$ are changed by the rule:
\begin{equation}\label{change2}
x_{old}^{2k'+k''+i}= x_{new}^{2k'+k''+i}+ U_i.
\end{equation} 
This   coordinate change does not affect the previous improvements; i.e.,  the structure of $\omega$ is still given  by \eqref{omega}, 
the terms $w_{\mu\nu}$  are still constant, the terms $v_{i\mu}$ are  still zero but now also the terms $u_{ij}$ are zero.  Thus,
the coordinates are  flat for $g$ and for $\omega$. This completes the proof of the Theorem.
\end{proof}

\newpage 
 
\section{The general case. } \label{sec:4} 

In this section, we consider a bilinear form $g+ \omega$  where both the symmetric and skew-symetric part may be degenerate.

\subsection{The  case  when the symmetric part is zero.\label{sec:4bis} }

We first  assume $\omega$ is  degenerate and   $g=0$,  and discuss  the existence of a flat coordinate system. 

\begin{theorem}\label{thm:darboux}   
There exists a smooth flat coordinate system for a given smooth skew-symmetric 2-form $\omega= \omega_{ij} = \sum_{i<j} \omega_{ij} dx^i \wedge dx^j$  if and only if  $\omega$ has constant rank and $d\omega=0$.
\end{theorem}  

Although this theorem is known, see e.g.     \cite[Theorem 5.1.3]{abraham}, we give a short proof  for selfcontainedness and because we  use  certain ideas of the proof  later.

\begin{proof}  
If  $\omega$ has maximal rank, then this result is the classical Darboux Theorem. Let us reduce  to it the case of smaller rank.
We denote by
\begin{equation} \label{eq:Romega} 
 \mathcal{R}_{\omega} = \{ v \in T_xM \mid \omega(v , \ \cdot )=0\}
\end{equation}
the kernel of $\omega$. Because $\omega$ has constant rank,  $\mathcal{R}_{\omega}$ is a smooth distribution. Furthermore, the condition  $d\omega=0$ implies that  it is integrable;  indeed, for any vector fields  $u, v\in \mathcal{R}_{\omega}$ and arbitrary vector field $w$  we have 
$$
  0=\mathcal{L}_{v} \left(\omega(u,w)\right) = \left(\mathcal{L}_{v} \omega\right)(u,w) + \omega([v,u], w) + \omega(u, [v,w]).  
$$
The first term on the right hand side vanishes because of the Cartan magic  formula, the third term because 
$u,w \in \mathcal{R}_\omega$. Then, $[v,u]\in \mathcal{R}_\omega$ and the distribution is integrable. 

\smallskip

Assume now the distribution has dimension $n-p$, where $p = \rank(\omega)$,  and 
consider a coordinate system $x^1,...,x^{p}$, $y^1,...,y^{n-p}$ such  that the distribution $\mathcal{R}_\omega$ is spanned by the vector fields  $\tfrac{\partial }{\partial y^1},..., \tfrac{\partial }{\partial y^{n-p}}$. In this coordinate system, we have
$$\omega= \sum_{i<j \le p} \omega_{ij}(x)\,  dx^i \wedge dx^j.$$ 
Since $d\omega = 0$, the components $\omega_{ij}$ do not depend on the variables $y^1,...,y^{n-p}$ \ 
(indeed, suppose for instance that   $\omega_{12}$ depends  on $y^1$ then   $d\omega$ would contain  the nonzero 
term $\tfrac{\partial \omega_{12}}{\partial y^1}  \,   dy^1\wedge dx^1\wedge dx^2$ which does not cancel with any other term).
The problem is then reduced to the classical Darboux Theorem  in dimension $p$, which completes the proof of the Theorem.
\end{proof}

\medskip

\begin{remark} \rm
It has been proved in \cite[Theorem 18]{BD}, see also  \cite[Theorem 14.1]{book}, that if $\omega$ is a symplectic form (that is non degenerate and close) of class  $C^{r,\alpha}$ with $r\in \mathbb{N}$ and $0<\alpha<1$, then the previous result still holds and  the obtained flat coordinates are of class $C^{r+1,\alpha}$. 

For a closed $2$-form of constant rank $<n$, one can still find flat coordinates of class $C^{r,\alpha}$. See  \cite[Theorem 3.2]{BDK}
and the extended discussion in \cite[\S 14.3]{book}. The degenerate case is proved by reducing it to the symplectic case, taking into account  that
factoring out the kernel of $\omega$ reduces one degree of regularity,
\end{remark}

\subsection{A necessary and sufficient  condition in  the general case. }

We consider the tensor field  $g_{ij} + \omega_{ij}$ with $g_{ij}$  symmetric and $\omega_{ij}$ skew-symmetric and study the existence of a flat coordinate system. This  is equivalent to the existence  of a  symmetric affine 
connection $\nabla= (\Gamma^i_{jk})$  such that its  curvature is zero and such that both $g$ and $\omega$ are parallel, meaning that 
\begin{eqnarray}
\frac{\partial g_{ij}}{\partial x^k} &= & \sum_s g_{sj} \Gamma^s_{ik} + g_{is} \Gamma^s_{jk}  \label{cond: symmetric}\\ 
\frac{\partial \omega_{ij}}{\partial x^k} &= & \sum_s \omega_{sj} \Gamma^s_{ik} + \omega_{is} \Gamma^s_{jk}.\label{cond: skewsymmetric} 
\end{eqnarray}

We view  (\ref{cond: symmetric}, \ref{cond: skewsymmetric})  as  a  linear inhomogeneous  system  of equations where the unknown quantities are the $\Gamma_{jk}^i$.   Algebraic compatibility conditions of each of the equations \eqref{cond: symmetric} and \eqref{cond: skewsymmetric} 
have  a clear geometric interpretation. Indeed,  as we understood in Section \ref{sec:proof1}, the algebraic consistency   condition of  \eqref{cond: symmetric} is \eqref{eq:1} and the freedom in choosing $\Gamma$ satisfying \eqref{cond: symmetric} once \eqref{eq:1} is satisfied    is the  addition of (possibly several expressions of the form) 
 \begin{equation} \label{eq:addition} 
v^iT_{jk} \ \ \textrm{ with $v\in \mathcal{R}_g$ and $T_{jk}= T_{kj}$.}
\end{equation} 
 
Concerning the second set of equations, we have the following  
\begin{lemma}
 Suppose $\omega$ is of class $C^1$, then  there exists $\Gamma_{ij}^k$ such that $\Gamma_{ij}^k = \Gamma_{ji}^k$ and  \eqref{cond: skewsymmetric} holds if and only if $\omega$ is a closed $2$-form.
\end{lemma}

\begin{proof} 
If $\omega$ is of class $C^{1, \alpha}$ for some $0< \alpha < 1$, then the Lemma immediately follows from Theorem \ref{thm:darboux}.
Since we only assume the $C^1$-regularity of $\omega$, a purely algebraic argument is needed. 
Observe first that a necessary condition is 
\begin{equation} \label{eq:comp:2}
   \frac{\partial \omega_{ij}}{\partial x^k} +  \frac{\partial \omega_{jk}}{\partial x^i}+ \frac{\partial \omega_{ki}}{\partial x^j}=0. 
\end{equation}
Indeed, if one relabels the index   in \eqref{cond: skewsymmetric}  by the schemes  $(i\to j \to k\to i)$ and  $(i\to k \to j\to i)$,  and add  the obtained equations to the initial equation,
one obtains  \eqref{eq:comp:2}. The geometric interpretation of  \eqref{eq:comp:2} is clear: it holds at every point if and only if $\omega$ is a closed form.

\smallskip

Observe now that, assuming  \eqref{eq:comp:2} holds, the system  \eqref{cond: skewsymmetric} is algebraically equivalent to the following system of linear equations:\footnote{In the symplectic case (when $\omega$ is nondegenerate) \eqref{eq:compsym} is known \cite{BCGR}}  
\begin{equation} \label{eq:compsym} 
\sum_s \omega_{is} \Gamma_{jk}^s  = \tfrac{1}{3} \left( \tfrac{\partial \omega_{ij}}{\partial x^k} +   
\tfrac{\partial \omega_{ik}}{\partial x^j}\right) + T_{ijk}, 
\end{equation}
where $T_{ijk}$ is totally symmetric.
This linear system   is always  compatible if \eqref{eq:comp:2} holds. Indeed, 
 the compatibility condition  for the equations \eqref{eq:compsym}  is  as follows: for any  $v\in \mathcal{R}_\omega$ the expression   
$$
\sum_s v^s\left( \tfrac{\partial \omega_{sj}}{\partial x^k} +   
\tfrac{\partial \omega_{sk}}{\partial x^j}\right)
$$  
should be symmetric in $j\longleftrightarrow k$. We see that this condition is  always fulfilled. 
We conclude that  \eqref{eq:comp:2} are sufficient conditions for compatibility of \eqref{cond: skewsymmetric}.   
\end{proof}

\medskip

Unfortunately, we do not have an easy geometric interpretation  for compatibility conditions of  the whole system  (\ref{cond: skewsymmetric},  \ref{cond: symmetric}).

\medskip

We now state our main result: 
\begin{theorem} \label{thm:3}
Let  $g+\omega$ be a smooth (here we assume $C^{\infty}$ for simplicity) bilinear form on a domain $U \subset \R^n$
(where $g$ is symmetric and $\omega$ is skew-symmetric). Suppose there is a flat  coordinate system for $g$ and $\omega$, then there exist   smooth functions $\Gamma^i_{jk}$ such that  both \eqref{cond: symmetric} and \eqref{cond: skewsymmetric} are fulfilled; in particular $\omega$ is closed and has constant rank.  Moreover, \eqref{eq:3} holds.  Conversely, if there exist  smooth functions $\Gamma^i_{jk}$   such that  \eqref{cond: symmetric} and  \eqref{cond: skewsymmetric} are fulfilled and  \eqref{eq:3} holds,   then there exists a flat coordinate system. 
\end{theorem}

\begin{proof} 
The direction ``$\Rightarrow$'' is clear. Indeed,  the conditions  \eqref{cond: symmetric} and \eqref{cond: skewsymmetric}   are geometric and are trivially satisfied  in a flat coordinate system for $\Gamma^i_{jk}=0$, therefore they hold in any coordinate system. Let us prove the non trivial direction.

\medskip

We assume the existence of smooth functions $\Gamma^i_{jk}$ defined on $U$,  such that \eqref{cond: symmetric} and \eqref{cond: skewsymmetric}  hold. We 
view these functions  as   coefficients of a  connection $\nabla$. The parallel transport with respect to  this connection preserves $g$ and $\omega$. 
In particular $g$  and $\omega$ have constant rank. We set $m = \rank(g)$ and $p = \rank(\omega)$. 
We also assume that  condition \eqref{eq:3} holds.  

\medskip

Our first step is to show that one may assume without loss of generality,  $\mathcal{R}_g\cap \mathcal{R}_\omega = \{0\}$ at one and therefore at every  point. Indeed, it is integrable and  we can consider a coordinate system $x^1,...,x^k,y^1,...,y^{n-k}$ such that $\mathcal{R}_g\cap \mathcal{R}_\omega$ is spanned by $\tfrac{\partial }{\partial y^1},..., \tfrac{\partial }{\partial y^{n-k}}.$ We know that both $g$ and $ \omega$ are preserved along the flow of any vector field $v\in \mathcal{R}_g\cap \mathcal{R}_\omega.$
Indeed, for $g$ we proved this  in Section \ref{sec:2}  and for $\omega$ in Section \ref{sec:4bis}.  Then, in the coordinate system $g$ and $\omega$ are given by 
$$
g= \sum_{i,j=1}^k g_{ij} dx^i dx^j   \   \  ,  \  \omega= \sum_{i<j\le k} \omega_{ij} dx^i \wedge dx^j
$$
such that $g_{ij}$ and $\omega_{ij}$  do not depend on the $y$-coordinates.  We see that the situation is reduced to an analogous situation on a $k$-dimensional manifold such that   $\mathcal{R}_g\cap \mathcal{R}_\omega$ is trivial. Note that     the existence of    smooth functions $\Gamma^i_{jk}$  satisfying  \eqref{cond: skewsymmetric} and \eqref{cond: symmetric} is not affected by this reduction since  the freedom  \eqref{eq:addition} with  $v \in  \mathcal{R}_g\cap \mathcal{R}_\omega $, affects  neither  \eqref{cond: symmetric}  nor \eqref{cond: skewsymmetric}.    
For the rest of the proof we may and will assume that  $\mathcal{R}_g\cap \mathcal{R}_\omega$ is trivial.  

\medskip 

Because of \eqref{cond: skewsymmetric}, the distribution $\mathcal{R}_\omega$ is integrable and invariant under parallel transport.
We assume  that $\omega$ has rank $p$, so  $\mathcal{R}_\omega$ has dimension $n-p$. Taking in account  \eqref{eq:3} and Theorem \ref{thm:1}  we obtain 
 the local existence of functions $f^1,...,f^{m}$, where $m = \rank(g)$ and such that the differentials $df^i$  are linearly independent and parallel, and 
\begin{equation} \label{eq:g}
  g= \sum_{i,j=1}^{m}  c_{ij} df^i df^j,
\end{equation} 
where $c = (c_{ij})$ is  a constant nondegenerate symmetric  $m\times m$ matrix. Without loss of generality, we may also assume that 
\begin{itemize}
\item [(a)]The   functions  $f^1,...,f^r$ have the property 
$\operatorname{Kernel}(df^i) \supseteq \mathcal{R}_\omega.$
 
\item[(b)] No nontrivial linear  combination   of the remaining  functions  $f^{r+1},...,f^{m} $   has  this property. 
\end{itemize} 

Indeed, if a function $f$ has property  $\nabla_i \nabla_j f =0$   at all points, then    the property   $\operatorname{Kernel}(df) \supseteq \mathcal{R}_\omega$  at one point $x$ implies  this property at all points. To see it, we chose a smooth path $c(t)$  joining a base point $x$ to an arbitrary point $y$ 
and denote by $v(t)  \in   \mathcal{R}_\omega(c(t))$   the parallel transport of the vector $v\in  \mathcal{R}_\omega(x)$ along this  curve. 
We then have
$$
 \frac{d}{dt} \left(df_{c(t)}(v(t))\right)  = \sum_k \tfrac{\partial }{\partial x^k} (df(v)) \  \tfrac{d c^k}{dt}= \sum_{s,k}  \left(v^s \nabla_k \nabla_s f  + \tfrac{\partial f}{\partial x^s} \nabla_k v^s\right)\tfrac{d c^k}{dt}=0 + 0 =0.
$$

Observe now that the hypothesis  $\mathcal{R}_g\cap \mathcal{R}_\omega = \{0\}$   implies that $r \leq p$ and  $n = p+m-r$. Furthermore the functions    $f^{r+1},...,  f^{m}$  restricted to any  integral submanifold of $\mathcal{R}_\omega$ define local coordinates on this submanifold. Indeed, no nontrivial linear  combination  of their differentials annihilates  $\mathcal{R}_\omega$.

\medskip

We  denote by $\hat{U} = U/ \mathcal{R}_\omega$ the quotient manifold of $U$ by the flow of all vector fields in $\mathcal{R}_\omega$ (we identify points of $U$  lying on the same integral submanifold of the distribution  $\mathcal{R}_\omega$). The manifold $\hat U$ is of dimension  $p = \rank (\omega)$, let us
fix some coordinates  $(z^1, \dots, z^{p})$ on $\hat{U}$ (concretely they are provided by any coordinate system on a manifold transverse to $\mathcal{R}_\omega$).

Observe that the  functions $f^1, \dots, f^r$ are constant on any integral manifold of $\mathcal{R}_\omega$ and therefore induce well defined functions on $\hat{U}$; we  denote them by  $\hat{f}^1, \dots, \hat{f}^r$.
Likewise, the form  $\omega$ induces a well defined $2$-form $\hat{\omega}$ on $\hat{U}$, which is clearly  a symplectic form on $\hat{U}$.  We denote by $\hat{P}$ the dual Poisson structure of $\hat{\omega}$. We claim that for any $1< \mu, \nu \leq r$, the Poisson bracket
$$
 \{\hat f^{\mu}, \hat f^{\nu}\} =  
  \sum_{i,j=1}^{p}  \hat{P}^{ij} \,  \frac{\partial \hat{f}^{\mu}}{ \partial z^j} \frac{\partial \hat{f}^{\mu}}{ \partial z^j}  
$$
is constant.  Indeed, this quantity is scalar and  constructed by linear algebraic operations from the  triple $(\omega, df^{\mu}, f^{\nu})$ (viewed now as objects on $U$)  and all the objects in this triple are parallel with respect  to $\nabla$.

\bigskip

We then know from Theorem \ref{th.gplussymplectic},  that there exists a coordinate system $y^1, \dots, y^{p}$ on $\hat{U}$ such that 
$y^j = \hat{f}^j$ for $j = 1, \dots, r$ and $\hat{\omega}$ has constant components in this coordinate. We thus have proved that the coordinate system   on $U$  defined by
$$
 (x^1, \dots, x^n)  = (f^1,\dots ,f^r, y^{r+1}, \dots, y^{p}, f^{r+1},\dots ,f^m)
$$
is flat for both $g$ and $\omega$.
\end{proof}

\bigskip

We conclude this section with a few remarks:

\begin{remark} \rm  
 \begin{enumerate}[(i)]
  \item  Let us stress  that verifying the hypothesis of  Theorem \ref{thm:3} requires only differentiation and  linear algebraic  operations. 
The main computational difficulty is to decide if the combined linear system containing  \eqref{cond: symmetric} and  \eqref{cond: skewsymmetric} is solvable. 
  \item  In the proof of Theorem \ref{thm:3} we assumed that all objects are as smooth as we need for the proof. 
We need them to be $C^{r,\alpha}$ with $r\ge 4$ and $0<\alpha<1$.  The flat coordinate system is then of class  $C^{r-3,\alpha}$. We do not have an example demonstrating that the regularity  is optimal, and in fact rather tend to believe that it is not optimal. 
       \newpage 
\item The proof of Theorem \ref{thm:3} shows that if $g$ has constant rank $1$, then there locally exists flat coordinates for  $g+\omega$   
if and only if the following conditions are satisfied:
\begin{enumerate}[(a)]
  \item $g = \pm \theta \otimes \theta$ for a closed $1$-form $\theta$.
  \item $\omega$ is closed and has constant rank.
  \item $\mathcal{R}_{\omega} \cap \mathcal{R}_{g}$ has constant dimension.
\end{enumerate}
\end{enumerate}
\end{remark}

\section{Ideas used in our proofs, conclusion and outlook.}  \label{outlook}

We solved,  for an arbitrary bilinear form,  the problem stated  by Riemann: we found   necessary and  sufficient conditions for a bilinear form to have constant entries in a local coordinate system. Our   results generalize  
the special cases   solved by Riemann  himself (when the bilinear form is symmetric and nondegenerate) and by Darboux (when it is skew-symmetric and nondegenerate).

Our  proofs in the smooth case  use   methods and, whenever possible,  notations  which were  available  to,  and used by, Riemann, Darboux and other fathers of differential geometry. These methods  include  basic real analysis,  basic  linear algebra  and   the standard results on the 
 existence and uniqueness  of systems of ordinary differential equations. 
 We also employ a  fundamental  idea used in particular by  Riemann in \cite{riemann2}, and which is  one of the main reasons for  many successful applications of  differential geometry   in mathematical physics:   
 \textit{if one works with geometric (covariant, in the language used in physics) objects, then one can   work with them in a coordinate system which is best adapted to the geometric situation}.

\smallskip 

The ideas behind the proofs are based  on concepts that  appeared  later.  Let us comment on  them and relate our  proofs to these concepts. 

\smallskip 

The first one is the concept of \textit{parallel transport}, it was introduced by Levi-Civita and was effectively used by Elie Cartan.  Recall that for any connection $\nabla = (\Gamma^i_{jk})$ the parallel transport along   the curve 
$c:[0,1]\to M$ is a linear  mapping $\tau_c:T_{c(0)}M \to T_{c(1)}M$.  It  it defined via the differential equation  $\sum_s \tfrac{dc^s(t)}{dt} \, \nabla_s  V^i(c(t))=0$  
 and can also be extended to arbitrary    tensors replacing the differential equation  by   $\sum_s \tfrac{dc^s(t)}{dt} \, \nabla_s  P^{i_1..i_k}_{j_1...j_m} (c(t))=0$.
The parallel  transport is compatible with all   geometric  operations on tensors.   

 \smallskip 
 
The condition that   a (possibly, degenerate) metric $g$ is parallel with respect a given connection  $\nabla= (\Gamma_{jk}^i)$ is equivalent to \eqref{eq:2}, and it  means that the parallel transport  preserves the  metric.   This  implies that the distribution $\mathcal{R}_g = \ker(g)$ is
invariant by parallel transport. It is then integrable and   the flow generated by any vector fields belonging to this distribution preserves  $g$
(in other words, the vector fields in  $\mathcal{R}_g$ are Killing vector fields). This  was a key argument  to reduce the proofs of Theorems \ref{thm:1}  to the  nondegenerate  case, which  was solved already by  Riemann.  

A similar reasoning  shows that in the situation discussed  in Theorem \ref{thm:3} one can ``quotient out'' first the joint kernel of $\omega$ and $g$ and then the kernel of  $\omega$, so the situation is reduced to the one discussed in Theorem \ref{thm:2}.  
Indeed,  the parallel transport preserves $\mathcal{R}_g$  ($ \mathcal{R}_\omega$, respectively) 
so the distributions of  $\mathcal{R}_g$ ($ \mathcal{R}_\omega$, respectively)  are integrable; moreover, $g$ ($ \omega$,respectively) is 
 preserved along the flow of any  vector fields lying in $\mathcal{R}_g$  ($ \mathcal{R}_\omega$,
respectively). This  allowed us to reduce the proofs   of \ref{thm:darboux} and Theorems \ref{thm:2}   to the  Darboux Theorem and to  Theorem \ref{thm:2}.  

\medskip 

The second concept is the idea of the  holonomy (group). This concept was successfully used already by Cartan and is still an active object of study. 
 For an affine connection $\nabla= (\Gamma^i_{jk})$ and a fixed point $p$, the holonomy group   generated by parallel transports along   curves $c:[0,1]\to {M}$ starting and ending at $p$ (the so-called loops).  The situation studied in Theorem \ref{thm:1} suggests that we consider the  holonomy group restricted to the anihilator
$$ \mathcal{R}^o(p) :=   \{\xi \in T^*_pM  \ \mid \  \textrm{Kernel} (\xi)\supseteq  \mathcal{R}_g(p)\}. $$ This space is  invariant with respect to  parallel transport along the loops since  it is 
defined via $\mathcal{R}_g$ which is parallel and therefore is invariant.   The Ambrose-Singer Theorem \cite{AS}, states that the holonomy group  is generated by the curvature  and is trivial if  the  curvature is zero. Now, \eqref{eq:3} implies that the curvature (of the connection  $\nabla$ viewed as the connection on the subbundle   $\mathcal{R}^o$ of $T^*M$) vanishes. This  implies the existence of sufficiently many  parallel 1-forms   belonging to this bundle. They are automatically closed and give rise to flat coordinates. 

\medskip

The third  concept came from the theory of integrable Hamiltonian   systems and was crystallized in the $1970$'s;  the standard references are  \cite{abraham,arnold}. The key observation is that for any two functions $f,h$ we have $[X_f, X_h]= X_{\{f,h\}}$, where $\{\ , \ \}$ is a Poisson structure and $X_h$ ,  $X_f$  are   the Hamiltonian vector fields corresponding to $f$ and $h$. 
The condition that $\{f,h\}$ is constant  implies  then  that vector fields $ X_f$ and $ X_h$ commute, which was the key point in the proof of Theorem \ref{thm:2}.

We have  mostly used the ``old-fashioned'' language and notations for two reasons. We wish  our proofs to be  available to any mathematician, even without special training in differential geometry and integrable systems.  Our declared goal is to present the proofs in the form the fathers of Riemannian  Geometry and Symplectic Geometry  could  understand them, and we believe that  we achieved this goal, at least partially. 
In addition, we expect  that  our results  may have applications outside of differential geometry.

The second reason  is that we aim at  understanding  the lowest regularity  assumptions  on $g$ and $\omega$ 
 under which  {our results}  holds. The ``modern'' differential geometrical ideas touched in this section require, as a rule, higher regularity than it is necessary. The point is that the so-called ``invariant notations'' that are highly successful in dealing with global differential geometry  on manifolds are, by nature, non-transparent   about regularity.

For example, the proof of Riemann works under the assumption that the metric is of class  $C^2$ (of course for Riemann himself  all functions were real analytic by definition). Later,  alternative proofs  appeared  which allowed to find the optimal regularity  assumption for the result of Riemann, see e.g.  \cite{BDMT,LeFloch, Cristinel,Mardare}.  Other examples include the Darboux theorem (under optimal regularity  assumptions it was proved in  \cite{BD} and  \cite{book}) and also the optimal  regularity  results for  isometries of Riemannian  (see e.g. the appendix to \cite{TM} for an overview)  and Finsler metrics \cite{lytchak,TM}.

As an illustration, our proof of Theorem \ref{thm:2} requires the bilinear forms to be of  rather high regularity, see Remark   \ref{rem:regularity}. 
By  contrast, the proof of  Theorem \ref{thm:2bis} produces flat coordinates of class $C^{r-2,\alpha}$.

\smallskip 

Though our results are local, they may open a door to a global investigation  of flat  bilinear forms. We already have several  relatively easy 
 global results, Corollaries \ref{cor:global} and \ref{cor:global1}. We also allow ourself to formulate the following conjecture: 
\begin{conjecture} \label{conj}
Suppose a closed manifold $M$ has a  flat (possibly degenerate) non-negative   definite metric $g$ of rank $m$.  Then, it is finitely covered by a manifold  which is diffeomorphic to a fiber bundle   over a $m$-dimensional torus. 
\end{conjecture} 
Note  that in the nondegenerate case $m = n = \dim M$, the   Conjecture  is equivalent to Bieberbach's Theorem, see e.g. \cite{Buser1985}.
In this situation  one can find $m$ parallel forms  $\theta_1,...,\theta_m$ on a finite cover $\tilde{M}$ of $M$
such that the lifted metric $\tilde g$ writes as  $g= \sum_{i,j}^m c_{ij} \theta_i \theta_j$, with a constant symmetric  positively definite matrix $c_{ij}$. 
Note also that by \cite{Miranda}, if a manifold admits $m$ closed forms such that in every point they are linearly independent, the manifold is diffeomorphic to  a  fibre bundle over a $m$-torus. 
 
\smallskip

Note also that some of our results can be easily generalized for the nonflat case. Say, one can define degenerate metrics of constant curvature  $\kappa\in \mathbb{R}$  by the   equation $R_{ijk\ell}=  \kappa (g_{i\ell} g_{jk} - g_{ik} g_{j\ell })$, 
and degenerate symmetric space by the formula $\nabla_m R_{ijk\ell}=0$. Neither   formula  depends on the freedom \eqref{eq:freedom}.

\section*{Acknowledgements}
The research of S. Bandyopadhyay is supported by the MATRICS research project grant (File No. MTR/2017/000414) titled \textquotedblleft On
the Equation $\left(  Du\right)  ^{t}A\,Du=G$ \& its Linearization, \&  Applications to Calculus of Variations\textquotedblright. V. Matveev and  M. Troyanov were supported in the framework of   the D-A-CH  cooperation  scheme; grants 200021L-175985 of the Swiss SNF and  MA 2565/6 of the DFG. V.  Matveev was also supported by MA 2565/7 of the DFG and   and  by ARC Discovery Programme   (grant DP210100951).

The authors are thankful to  Pierre Bieliavsky  for useful remarks and suggestions.

\end{document}